\documentclass[letterpaper]{scrreprt}

\RequirePackage[utf8]{inputenc}
\RequirePackage{amsfonts}
\RequirePackage{amssymb}
\RequirePackage{mathtools}
\newcommand\eqdef{=\vcentcolon}

\usepackage[english]{babel}
\usepackage{amsthm}
\usepackage{mathtools}
\usepackage{epigraph}
\usepackage{xparse}
\usepackage[autostyle = true, english = american, strict = true, autopunct=true]{csquotes}
\usepackage[final, babel]{microtype}
\usepackage{xspace}
\usepackage[inline]{enumitem}
\usepackage{authblk} %
\usepackage{enumitem}
\setlist[description]{leftmargin=\parindent,labelindent=\parindent}
\usepackage{ebproof} %
\usepackage{tikz}
\usetikzlibrary{decorations.pathreplacing,automata,positioning,arrows,calc}
\tikzset{every picture/.style={line width=0.6pt}, rounded corners=3pt}
\usepackage[tikz]{bclogo} %
\usepackage{hyperref}

\definecolor{persimmon}{rgb}{0.93, 0.35, 0.0}
\definecolor{cssgreen}{rgb}{0.0, 0.5, 0.0}
\hypersetup{plainpages=false, bookmarksnumbered=true, bookmarksopen=true, bookmarksdepth=2, colorlinks=true, breaklinks=true, linkcolor=persimmon, urlcolor=blue, citecolor=cssgreen, pdfdisplaydoctitle= true}
\usepackage[hyperref=true, natbib=true, backend=bibtex8, block=none, texencoding=utf-8, useprefix=true, block=space, sorting=nyt, abbreviate=false]{biblatex}
\author{Clément Aubert%
	\thanks{\texttt{e-mail: \href{mailto:caubert@augusta.edu}{caubert@augusta.edu}}. Some of this work was done when I was supported by the NSF grant 1420175 and collaborating with Patricia Johann, \url{http://www.cs.appstate.edu/~johannp/}.}
}
\theoremstyle{definition}
\newtheorem{theorem}{Theorem}

\newtheorem{lemma}{Lemma}
\newtheorem{definition}{Definition}

\newtheorem{remark}{Remark}

\addto\extrasenglish{

}

\DeclareMathOperator{\op}{op}

\DeclareMathOperator{\id}{id}

\DeclareMathOperator{\Obj}{Obj}
\DeclareMathOperator{\ev}{ev}
\DeclareMathOperator{\eval}{ev}

\newcommand{\cat}[1]{\textbf{#1}}

\newcommand{\natt}{\xrightarrow{\bullet}} %
\newcommand{\ccat}[1]{\mathbb{#1}} %

\newcommand{\forg}{\mathsf{forg}}
\DeclareMathOperator{\Hom}{Hom}
\DeclareMathOperator{\AHom}{AHom}
\newcommand{\dst}{\gamma} %
\newcommand{\one}{\mathbf{1}}
\newcommand{\zero}{\mathbf{0}}
\newcommand{\expo}{\Rightarrow}
\newcommand{\dupl}{\delta}
\newcommand{\oeq}{\mathsf{E}}
\newcommand{\meq}{\mathsf{e}}
\newcommand{\mm}{\multimap}
\newcommand{\texpo}{\overset{*}{\expo}}

\newcommand{\assoc}{\alpha}
\newcommand{\lunitor}{\lambda}
\newcommand{\runitor}{\rho}
\newcommand{\sym}{s}

\newcommand{\monad}[1]{\mathcal{#1}}
\newcommand{\unit}{\eta}
\newcommand{\mult}{\mu}
\DeclareMathOperator{\lst}{lst} %
\DeclareMathOperator{\rst}{rst} %
\DeclareMathOperator{\sst}{sst} %

\newcommand{\liftkl}[1]{#1^{\#}}

\DeclareMathOperator{\codom}{cod}

\newcommand{\is}[1]{\underline{#1}}
\newcommand{\slic}[2]{\ccat{#1} {/} #2} %
\newcommand{\coslic}[2]{#1 \backslash \ccat{#2}} %

\newcommand{\pull}[1]{#1^{\Delta}}%
\newcommand{\comp}[1]{\Sigma_{#1}} %

\newcommand{\pullra}[1]{\Pi_{#1}} %
\newcommand{\counit}{\epsilon}

\newcommand{\reind}[1]{{#1}^*} %
\newcommand{\opreind}[1]{{#1}_*} %

\newcommand{\alg}[1]{\bar{#1}}

\allowdisplaybreaks

\title{Categories for Me, and You?\footnote{The title echoes the notes of Olivier Laurent, available at \url{https://perso.ens-lyon.fr/olivier.laurent/categories.pdf}.}}
\makeatletter%
\begin{document}
\maketitle

\epigraph{This result is \textbf{folklore}, which is a technical term for a method of publication in category theory. It means that someone sketched it on the back of an envelope, mimeographed it (whatever that means) and showed it to three people in a seminar in Chicago in 1973, except that the only evidence that we have of these events is a comment that was overheard in another seminar at Columbia in 1976. Nevertheless, if some younger person is so presumptuous as to write out a proper proof and attempt to publish it, they will get shot down in flames.}{\href{http://math.andrej.com/2012/09/28/substitution-is-pullback/comment-page-1/\#comment-21991}{Paul Taylor}}

\tableofcontents

\chapter*{Disclaimers}

\section*{Purpose}

Those notes are an expansion of a document whose first purpose was to remind myself the following two equations:\footnote{And yet somehow they failed with that respect, until Fredrik Nordvall Forsberg kindly pointed out that I swapped the definition of mono and epi (and of terminal and initial) in the first version of that document!}

\blockquote{Mono = injective = faithful\\Epi = surjective = full}

I am \emph{not} an expert in category theory, and those notes should \emph{not} be trusted\footnote{That being said, those notes were carefully written, and \emph{precise} references are given when available.}.
However, if it happens that someone can save the time that was lost tracking the definition of locally cartesian closed category (\autoref{def:lcc}), of the cartesian structure in slice categories (\autoref{sec:cartesian-structure-in-slice}), or of the \enquote{pseudo-cartesian structure} on Eilenberg–Moore categories (\autoref{sec:em}), then those notes will have fulfilled their goal of giving to those not present in that seminar in Chicago in 1973 a chance to find a proper definition, and detailed proofs.

I'm not intending to be as presumptuous as to try to publish those notes, but plan on continuing on tuning them as I see fit.

\section*{Conventions}

Those notes are not self-contained (for instance, the definition of \enquote{commuting diagram} is supposed to be known), but they are aiming at being as uniform as possible.
The references point to either the most simple and accessible description (in the case e.g.\ of the definition of binary product) or to the only known reference.
When a structure is known under different names, they are listed.
Some of the subscripts are dropped when they can be inferred from context.

In Category Theory, there are as many notational conventions as \rule{1.1cm}{0.4pt} (fill in the blank), but the following one will be used:

\begin{center}
	\begin{tabular}{r l l}
		Objects                     & \(A\), \(B\), \(C\), \(D\), \(E\), \(I\), \(J\), \(P\), \(X\), \(Y\), \(Z\)                         \\
		Morphisms                   & \(e\), \(f\), \(g\), \(h\), \(k\), \(m\), \(p\), \(v\)                                              \\
		Categories                  & \(\ccat{A}\), \(\ccat{B}\), \(\ccat{C}\), \(\ccat{D}\), \(\ccat{E}\)                                \\
		Functors                    & \(F\), \(G\), \(T\), \(U\)                                                                          \\
		Natural Transformations     & \(\alpha\)                                                                                          \\
		Monads                      & \(\monad{T}\)                                                                                       \\
		Object in Slice Category    & \((X, f_X)\), \((Y, f_Y)\), \((A, f_A)\)                                    & (\autoref{sec:slice}) \\
		Morphisms in Slice Category & \(\is{f}\), \(\is{g}\), \(\is{k}\), \(\is{l}\), \(\is{m}\)                  & (\autoref{sec:slice}) \\
		\(\monad{T}\)-algebras      & \(\alg{A} = (A, f_A)\), \(\alg{B} = (B, f_B)\)                              & (\autoref{sec:em})    \\
	\end{tabular}
\end{center}

Sometimes, those symbols will be sub- or superscripted with symbols, such as number or object's name, for the sake of clarity.

\chapter{On Categories, Functors and Natural Transformations}

\section{Basic Definitions}

\begin{definition}[Category]
	\label{def:cat}
	A category \(\ccat{C}\) consists of
	\begin{description}[style=unboxed]
		\item[a class of objects (or elements)] denoted \(\Obj(\ccat{C})\),
		\item[a class of morphisms (or arrows, maps)] between the objects, denoted \(\Hom_{\ccat{C}}\).
	\end{description}
	For a particular morphim \(f\), we write \(f : A \to B\) if \(A\) and \(B\) are objects in \(\ccat{C}\), call \(A\) (resp.\ \(B\)) the \emph{domain} (resp.\ the \emph{co-domain}) \emph{of \(f\)} and write \(\Hom_{\ccat{C}}(A,B)\) (or \(\textrm{Mor}_{\ccat{C}}(A,B)\), \(\ccat{C}(A,B)\)) for the collection of all the morphisms in \(\ccat{C}\) between \(A\) and \(B\).
	For every three objects \(A\), \(B\) and \(C\) in \(\ccat{C}\), the \emph{composition of \(f : A \to B\) and \(g : B \to C\)} is written as \(g \circ f\) (or \(gf\), \(f;g\), \(fg\)), and its domain (resp.\ co-domain) is \(A\) (resp.\ \(C\)).

	The classes of objects and of morphisms, together with the definition of composition, should be such that the following holds:
	\begin{description}[style=unboxed]
		\item[Associativity] for every \(f : A \to B\), \(g : B \to C\) and \(h : C \to d\), \(h \circ (g \circ f) = (h \circ g) \circ f\)
		\item[Identity] for every object \(A\), there exists a morphism \(\id_A: A \to A\) called the identity morphism for \(A\), such that for every morphism \(f : B \to A\) and every morphism \(g : A \to C\), we have \(\id_A \circ f = f\) and \(g \circ \id_A = g\).
	\end{description}
\end{definition}

Composition and identity can often be inferred from the classes of objects and morphisms, and will be left implicit when this is the case.

The notion of \emph{isomorphism}, written \(\cong\) and used in the two following definitions, is formally introduced in \autoref{def:prop-of-morphisms}.

\begin{definition}[Functors]
	Let \(\ccat{C}\) and \(\ccat{D}\) be two categories, a morphism\footnote{%
		Using the word \enquote{morphism} in the technical sense of \autoref{def:cat} would require to observe that categories and their functors form a category -- which is true --. We use this term here informally, sometimes \enquote{mapping} or \enquote{map} is used to avoid confusing the formal definition of morphism with the informal notion of \enquote{not necessarily structure-preserving relationship between two mathematical objects}.
	} \(F : \ccat{C} \to \ccat{D}\) is
	\begin{description}[style=unboxed]
		\item[a pseudo-functor] if \(\forall A, B, C \) in \(\ccat{C}\), \(\forall f : A \to B, g : B \to C\) in \(\ccat{C}\),
		      \begin{align}
			      F x_i         & \text{ is in } \ccat{D}                    \\
			      F \id_{x_i}   & \cong \id_{F x_i} \label{pseudo-fun-1}     \\
			      F (g \circ f) & \cong F(g) \circ F(f) \label{pseudo-fun-2}
		      \end{align}
		\item[a functor] if it is a pseudo-functor, \ref{pseudo-fun-1} and \ref{pseudo-fun-2} are equalities.
	\end{description}
\end{definition}

\begin{definition}[Natural transformation~{\cite[16]{Maclane1971}}]
	Given \(F, G : \ccat{D} \to \ccat{C}\) two functors, \emph{a natural transformation} \(\alpha : F \natt G\) assigns to every object \(A \) in \( \ccat{D}\) a morphism \(\alpha_A : FA \to G A\) in \(\ccat{C}\) such that \(\forall f : A \to B\) in \(\ccat{D}\), the following commutes in \(\ccat{C}\):

	{\centering

	\tikz{
		\node (FA) at (0, 0) {\(FA\)};
		\node (GA) at (0, -2) {\(GA\)};
		\node (FB) at (2, 0) {\(FB\)};
		\node (GB) at (2, -2) {\(GB\)};
		\draw [->] (FA) to node[above]{\(Ff\)} (FB);
		\draw [->] (GA) to node[above]{\(Gf\)} (GB);
		\draw [->] (FA) to node[left]{\(\alpha_A\)} (GA);
		\draw [->] (FB) to node[right]{\(\alpha_{B}\)} (GB);
	}

	}
	We then say that \(\alpha_A : FA \to FB\) is \emph{natural in \(A\)}.

	If, for every object \(A\), the morphism \(\alpha_A\) is an isomorphism, then \(\alpha\) is said to be a \emph{natural isomorphism} (or a natural equivalence, an isomorphism of functors).
\end{definition}

Since \(A\) is universally quantified, we simply write that \(\alpha\) is natural, and remove the \(A\) from the previous diagram.
Even if the \(\natt\) notation is convenient to distinguish natural transformations from functors and morphisms, we will omit it most of the time, and use \(\to\) for natural transformation, trusting the reader to understand whenever we are refering to a natural transformation or some other construction.

\section{Properties of Morphisms, Objects, Functors, and Categories}

\begin{definition}[Properties of morphisms]
	\label{def:prop-of-morphisms}
	Let \(F : \ccat{C} \to \ccat{D}\) be a functor, a morphism \(f : X \to Y\) in \(\ccat{C}\) is
	\begin{description}[style=unboxed]
		\item[an epimorphism (or onto, right-cancellative)] if for all \(g_1, g_2 : Z \to X \), \(g_1 \circ f = g_2 \circ f \implies g_1 = g_2\).
		      We write \(\twoheadrightarrow \).
		\item[a monomorphism (or left-cancellative)] if for all \(g_1, g_2 : Z \to X \), \(f \circ g_1 = f \circ g_2 \implies g_1 = g_2\).
		      We write \(\rightarrowtail \) or \(\hookrightarrow\) (but this last one is often reserved for inclusion morphisms).
		\item[a bimorphism] if it is a monomorphism and an epimorphism.
		      We write \(\xrightarrow{\sim}\).
		\item[a retraction (has a right inverse)] if there exists \(g : Y \to X \) such that \(f \circ g = \id_Y \).
		      Then \(g\) is a \emph{section} of \(f\).
		\item[a section (has a left inverse)] if there exists \(g : Y \to X \) such that \(g \circ f = \id_X \).
		      Then \(g\) is a \emph{retraction} of \(f\).
		\item[an isomorphism] if there exists \(g : Y \to X \) such that \(f \circ g = \id_Y\) and \(g \circ f = \id_X\).
		      We write \(\cong\), and \(f^{-1}\) for \(g\).
		\item[an endomorphism] if \(X = Y\).
		\item[an automorphism] if it is both an isomorphism and an endomorphism.
		\item[over \(u\) in \(\ccat{D}\)] if \(F f = u\)
		\item[cartesian over (or above) \(u : I \to J\) in \(\ccat{D}\) (or a cartesian, or terminal, lifting of \(u\))] if \(F f = u\) and for all \(Z\), for all \(g : Z \to Y\) in \(\ccat{C}\) for which \(F g = u \circ w\) for some \(w: F Z \to I \), there is a unique \(h : Z \to X\) in \(\ccat{C}\) such that \(F h = w \) and \(f \circ h = g\).

			      
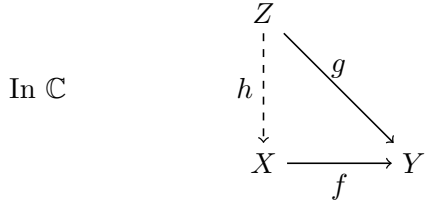
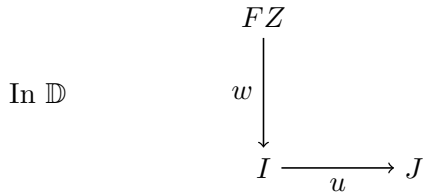
{\centering

				      \tikz{
					      \node (catC) at (-3, -1) {In \(\ccat{C}\)};
					      \node (Z) at (0, 0) {\(Z\)};
					      \node (X) at (0, -2) {\(X\)};
					      \node (Y) at (2, -2) {\(Y\)};
					      \draw [->, dashed] (Z) to node[left] {\(h\)} (X);
					      \draw [->] (X) to node[below]{\(f\)} (Y);
					      \draw [->] (Z) to node[above]{\(g\)} (Y);
					      \begin{scope}[shift={(0,-4)}]
						      \node (catD) at (-3, -1) {In \(\ccat{D}\)};
						      \node (FZ) at (0, 0) {\(FZ\)};
						      \node (I) at (0, -2) {\(I\)};
						      \node (J) at (2, -2) {\(J\)};
						      \draw [->] (FZ) to node[left] {\(w\)} (I);
						      \draw [->] (I) to node[below]{\(u\)} (J);
					      \end{scope}
				      }

			      }
		      We write \(u^{\S}_X\) for the cartesian morphism over \(u\) with codomain \(X\).
		      For a reason that will become clear with \autoref{def:re-indexing}, we write \(u^*X\) for the domain of \(u^{\S}(X)\).

		      It used to be the case that cartesian morphisms were called \enquote{strong cartesian}, the qualification of \enquote{cartesian} being reserved for the case where \(w = \id_{I}\) \cite[Appendix B]{Winskel1995}.

		\item[cartesian] if it is cartesian over \(F f \).
		\item[opcartesian over (or above) \(u : I \to J\) in \(\ccat{D}\)] if \(F f = u\) and for all \(Z\), for all \(g : X \to Z\) in \(\ccat{C}\) for which \(F g = w \circ u \) for some \(w: J \to F Z \), there is a unique \(h : Y \to Z\) in \(\ccat{C}\) such that \(F h = w \) and \(h \circ f= g\).

			      
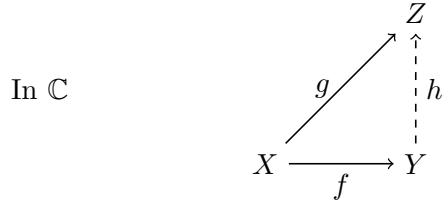
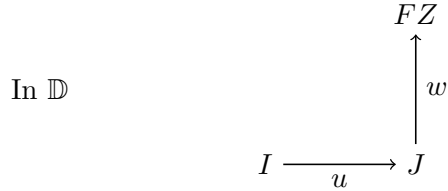
{\centering

				      \tikz{
					      \node (catC) at (-3, -1) {In \(\ccat{C}\)};
					      \node (Z) at (2, 0) {\(Z\)};
					      \node (X) at (0, -2) {\(X\)};
					      \node (Y) at (2, -2) {\(Y\)};
					      \draw [->, dashed] (Y) to node[right] {\(h\)} (Z);
					      \draw [->] (X) to node[below]{\(f\)} (Y);
					      \draw [->] (X) to node[left]{\(g\)} (Z);
					      \begin{scope}[shift={(0,-4)}]
						      \node (catD) at (-3, -1) {In \(\ccat{D}\)};
						      \node (FZ) at (2, 0) {\(FZ\)};
						      \node (I) at (0, -2) {\(I\)};
						      \node (J) at (2, -2) {\(J\)};
						      \draw [->] (J) to node[right] {\(w\)} (FZ);
						      \draw [->] (I) to node[below]{\(u\)} (J);
					      \end{scope}
				      }

			      }

		      We write \(u_{\S}^X\) for the opcartesian morphism over \(u\) with domain \(X\).
		      For a reason that will become clear with \autoref{def:opre-indexing}, we write \(u_*X\) for the co-domain of \(u_{\S}(X)\).
		\item[vertical] if \(F f =\id_{Fx} \).
	\end{description}
	The terms over, cartesian over, opcartesian over and vertical are mostly used when \(F\) is a fibration (\autoref{def:fibration}).
\end{definition}

\begin{definition}[Properties of objects]
	An object \(A \) in \( \ccat{C}\) is
	\begin{description}[style=unboxed]
		\item[terminal (or final)] if for all \(B\) in \( \ccat{C} \), there exists a unique morphim \(f: B \to A\).
		      Such an object is denoted \(\one\) (or \(t\)) and is unique, and the unique morphism \(A \to \one\) is denoted \(!_A\).
		\item[initial (or co-terminal, universal)] if for all \(B\) in \( \ccat{C} \), there exists a unique morphim \(f: A \to B\).
		      Such an object is denoted \(\zero\) (or \(i\)) and is unique, and the unique morphism \(0 \to A\) is denoted \(!^A\).
		\item[strict initial] if it is initial and every morphism \(f : B \to A\) is an isomorphism.
		\item[zero (or null)] if it is both initial and terminal.
	\end{description}
\end{definition}

\begin{definition}[Properties of categories]
	\label{def-prop-cat}
	A category \(\ccat{C}\) has
	\begin{description}[style=unboxed]
		\item[(cartesian binary) product {\cite[35]{Aluffi2009}}] if \(\forall A_1, A_2 \) in \( \ccat{C}\), \(\exists B\) in \( \ccat{C}\) and \(\exists \pi_i : B \to A_i\) for \(i \in \{1, 2\}\), such that \(\forall f_i : D \to A_i\), \(\exists ! v : D \to B\) such that the following commutes:

		      
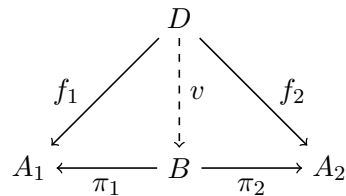
{\centering

		      \tikz{
			      \node (d) at (0,0) {\(D\)};
			      \node (c) at (0, -2) {\(B\)};
			      \node (c1) at (-2, -2) {\(A_1\)};
			      \node (c2) at (2, -2) {\(A_2\)};
			      \draw [->] (d) to node[left = 0.2cm]{\(f_1\)} (c1);
			      \draw [->] (d) to node[right = 0.2cm]{\(f_2\)} (c2);
			      \draw [->, dashed] (d) to node[right] {\(v\)} (c);
			      \draw [->] (c) to node[below]{\(\pi_1\)} (c1);
			      \draw [->] (c) to node[below]{\(\pi_2\)} (c2);
		      }

		      }

		      We call \(B\) the \emph{product of \(A_1\) and \(A_2\)}, and denote it with \(A_1 \times A_2\), \(v\) is the \emph{product of the morphisms \(f_1\) and \(f_2\)} and is written \(f_1 \times f_2\), and \(\pi_i\) are the \emph{(canonical) projections}.

		      For all \(A\) in \( \ccat{C}\), a morphism \(\dupl_A : A \to A \times A\) (sometimes written \(\Delta_A\)) is \emph{a diagonal morphism} if, for \(i \in \{1, 2\}\), \(\pi_i \circ \dupl_A = \id_A\).
		      Moreover, for all \(f : A \to B \) and \(g : A \to C\), we write \(\langle f, g\rangle\) for \((f \times g) \circ \dupl_A : A \to (B \times C)\).

		\item[all finite product] if it has all (cartesian binary) product and a terminal object~\cite[Definition 2.19]{Awodey2010}.
		\item[exponent] if \(\ccat{C}\) has (cartesian binary) product and \(\forall A_1, A_2 \) in \( \ccat{C}\), \(\exists B\) in \( \ccat{C}\) and \(f : B \times A_2 \to A_1 \) such that \(\forall C\) and \(g : C \times A_2 \to A_1\), \(\exists ! u : C \to B\) such that the following commutes:

		      {\centering

		      \tikz{
			      \node (d) at (-2,0) {\(C\)};
			      \node (c) at (-2,-2) {\(B\)};
			      \draw [->, dashed] (d) to node[left] {\(u\)} (c);
			      \node (dc2) at (0, 0) {\(C \times A_2\)};
			      \node (cc2) at (0, -2) {\(B \times A_2\)};
			      \node (c1) at (2, -2) {\(A_1\)};
			      \draw [->, dashed] (dc2) to node[left] {\(u \times \id_{A_2}\)} (cc2);
			      \draw [->] (cc2) to node[below]{\(f\)} (c1);
			      \draw [->] (dc2) to node[above]{\(g\)} (c1);
		      }

		      }

		      Then,
		      \begin{itemize}
			      \item \(B\) is \emph{an exponential object}, denoted \(A_2 \expo A_1\) (or \(A_1^{A_2}\)),
			      \item \(u\) is the \emph{transpose of \(g\)}, denoted \(\lambda g\) (or \(\tilde{g}\)),
			      \item \(f\) is \emph{the evaluation morphism}, denoted \(\eval_{A_1, A_2}\).
		      \end{itemize}
		      and we say that
		      \begin{itemize}
			      \item \(B\) is \emph{exponentiating} if \(\forall A_1\) in \( \ccat{C}\), \(A_1 \expo B\) exists,
			      \item \(B\) is exponentiable (or powerful) if \(\forall A_1 \) in \( \ccat{C}\), \(B \expo A_1\) exists.
		      \end{itemize}

		\item[pullback] if, for \(i \in \{1, 2\}\), \(\forall A_i, B\) in \( \ccat{C}\), \(f_i : A_i \to B\), there exists an unique \(C\) in \( \ccat{C}\), \(p_i : C \to A_i\) such that the following commutes:

		      {\centering

		      \tikz{
			      \node (x) at (0,2) {\(C\)};
			      \node (c1) at (0,0) {\(A_1\)};
			      \node (c2) at (2,2) {\(A_2\)};
			      \node (c) at (2,0) {\(B\)};
			      \draw [->] (x) to node[left]{\(p_1\)} (c1);
			      \draw [->] (x) to node[above]{\(p_2\)} (c2);
			      \draw [->] (c1) to node[below]{\(f_1\)} (c);
			      \draw [->] (c2) to node[right]{\(f_2\)} (c);
		      }

		      }

		      and such that for all \(D\), \(g_i : D \to A_i\) such that \(f_1 \circ g_1 = f_2 \circ g_2\), there exists a unique \(v : D \to C\) such that \(g_i = p_i \circ v\):

		      
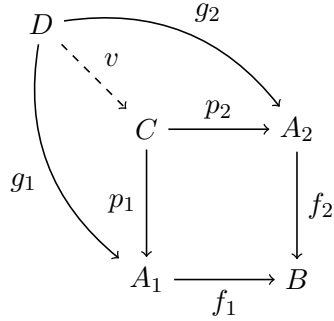
{\centering

		      \begin{tikzpicture}[node distance=2cm, auto]
			      \node (x) {\(C\)};
			      \node (c2) [right of=x] {\(A_2\)};
			      \node (c1) [below of=x] {\(A_1\)};
			      \node (c) [below of=c2] {\(B\)};
			      \node (z) [node distance=1.4cm, left of=x, above of=x] {\(D\)};
			      \draw[->] (x) to node {\(p_2\)} (c2);
			      \draw[->] (x) to node [swap] {\(p_1\)} (c1);
			      \draw[->] (c1) to node [swap] {\(f_1\)} (c);
			      \draw[->] (c2) to node {\(f_2\)} (c);
			      \draw[->, bend right] (z) to node [swap] {\(g_1\)} (c1);
			      \draw[->, bend left] (z) to node {\(g_2\)} (c2);
			      \draw[->, dashed] (z) to node {\(v\)} (x);
		      \end{tikzpicture}

		      }
		      This diagram is called the \emph{pullback diagram} (or \emph{cartesian square}), and we usually draw a right angle in the corner where \(C\) is, as follows:

		      
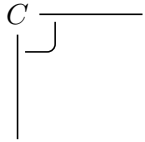
{\centering

		      \begin{tikzpicture}[node distance=1.8cm, auto]%
			      \node (x) {\(C\)};
			      \node (c2) [right of=x] {};
			      \node (c1) [below of=x] {};
			      \draw[-] (x) to node {} (c2);
			      \draw[-] (x) to node [swap] {} (c1);
			      \draw[-] ($(x)+(0.1, -0.5)$) -- ($(x)+(0.5, -0.5)$) -- ($(x)+(0.5, -0.1)$);
		      \end{tikzpicture}

		      }

		      The object \(C\) is sometimes called
		      \begin{itemize}
			      \item \emph{the fibred product of \(A_1\) and \(A_2\) over \(B\)} and written \(A_1 \times^{B} A_2\) (or \(A_1 \times_{\ccat{C}}^{c} A_2\)),
			      \item \emph{the pullback of \(A_2\) along \(f_1\)} and written \(f_1 A_2\),
			      \item \emph{the pullback of \(A_1\) along \(f_2\)} and written \(f_2 A_1\).
		      \end{itemize}
		      The morphism \(p_1\) (resp.\ \(p_2\)) is sometimes called \emph{the pullback of \(f_2\) along \(f_1\)} (resp.\ \emph{the pullback of \(f_2\) along \(f_1\)}) and written \(f_2^* f_1\) (resp.\ \(f_1^* f_2\)).
		\item[pushout] if \(\ccat{C}^{\op}\) (cf. \autoref{def:construction_cat}) has pullback.
		\item[equalizers] if for all \(f, g: A \to B\), there exists an object \(\oeq_{f, g}\) and a morphism \(\meq_{f,g} :\oeq_{f, g} \to A\) such that \(f \circ \meq_{f, g} = g \circ \meq_{f, g}\) , and such that for all object \(Z\) and morphism \(m : Z \to A\) such that \(f \circ m = g \circ m\), there exists a unique \(u : Z \to \oeq_{f, g}\) such that \(\meq_{f, g} \circ u = m\).

			      
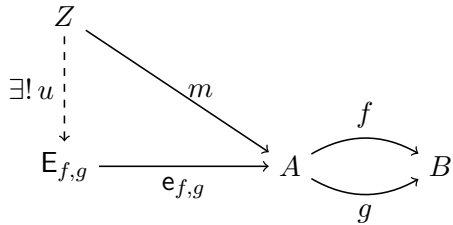
{\centering

				      \begin{tikzpicture}
					      \node (C) at (0, 2) {\(Z\)};
					      \node (eq) at (0, 0) {\(\oeq_{f,g}\)};
					      \node (AtoB) at (3, 0) {\(A\)};
					      \node (TAtoB) at (5, 0) {\(B\)};
					      \draw[->, bend left] (AtoB) to node[above]{\(f\)}(TAtoB);
					      \draw[->, bend right] (AtoB) to node[below]{\(g\)}(TAtoB);
					      \draw[->, dashed] (C) to node[left]{\(\exists ! \,u\)} (eq);
					      \draw[->] (eq) to node[below]{\(\meq_{f,g}\)} (AtoB);
					      \draw[->] (C) to node[right]{\(m\)} (AtoB);
				      \end{tikzpicture}

			      }

	\end{description}
	A category \(\cat{C}\) is
	\begin{description}[style=unboxed]
		\item[the category \(\ccat{1}\)] if it has only one object (often written \(\ccat{1}\) as well) and one morphism.
		\item[monoidal] if it has
		      \begin{itemize}
			      \item a bifunctor \(\otimes: \ccat{C}\times\ccat{C} \to \ccat{C}\),
			      \item a neutral object \(I\) (a right and left identity),
			      \item natural isomorphisms
			            \begin{itemize}
				            \item \(\assoc_{A,B,C} : (A \otimes B) \otimes C \to A \otimes (B \otimes C)\) (the \emph{associator}),
				            \item \(\lunitor_{A} : I \otimes A \to A\) (the \emph{left unitor}),
				            \item \(\runitor_A : A \otimes I \to A\) (the \emph{right unitor})
			            \end{itemize}
		      \end{itemize}
		      such that for all \(A, B, C\) and \(D\), the following diagrams commute:

		      {\centering

		      \tikz%
		      {
			      \node (TL) at (0, 2) {\(((A \otimes B) \otimes C) \otimes D\)};
			      \node (TM) at (0, 0) {\((A \otimes (B \otimes C)) \otimes D\)};
			      \node (TR) at (0, -2) {\(A \otimes ((B \otimes C) \otimes D)\)};
			      \node (BL) at (7, 2) {\((A \otimes B) \otimes (C \otimes D) \)};
			      \node (BR) at (7, -2) {\(A \otimes (B \otimes (C \otimes D))\)};
			      \draw [->] (TL) to node[right]{\(\assoc_{A,B,C} \otimes \id_D\)} (TM);
			      \draw [->] (TM) to node[right]{\(\assoc_{A, B \otimes C, D}\)} (TR);
			      \draw [->] (TL) to node[above]{\(\assoc_{A \otimes B, C, D} \)} (BL);
			      \draw [->] (BL) to node[right]{\(\assoc_{A, B, C \otimes D}\)} (BR);
			      \draw [->] (TR) to node[below]{\(\id_{A} \otimes \assoc_{B, C, D}\)} (BR);
		      }

		      \tikz{
			      \node (TL) at (0,2) {\((A \otimes I) \otimes B\)};
			      \node (TR) at (4,2) {\(A \otimes (I \otimes B) \)};
			      \node (BM) at (2,0) {\(A \otimes B\)};
			      \draw [->] (TL) to node[above]{\(\assoc_{A,I,B}\)} (TR);
			      \draw [->] (TL) to node[left]{\(\runitor_{A} \otimes \id_B\)} (BM);
			      \draw [->] (TR) to node[right = 0.3cm]{\(\id_{A} \otimes \lunitor_{B}\)} (BM);
		      }

		      }

		\item[strict monoidal]
		      if it is monoidal and \(\assoc\), \(\lunitor\) and \(\runitor\) are identities,
		\item[symmetric monoidal]
		      if it is monoidal and it have an isomorphism \(\sym_{A,B} : A \otimes B \to B \otimes A\) such that the following three diagrams commute:

		      {\centering

		      \tikz{
			      \node (TL) at (0,2) {\(A \otimes I\)};
			      \node (TR) at (4,2) {\(I \otimes A\)};
			      \node (BM) at (2,0) {\(A\)};
			      \draw [->] (TL) to node[above]{\(\sym_{A,I}\)} (TR);
			      \draw [->] (TL) to node[left]{\(\runitor_{A}\)} (BM);
			      \draw [->] (TR) to node[right = 0.3cm]{\(\lunitor_{A}\)} (BM);
		      }
		      \tikz{
			      \node (TL) at (0,0) {\(A \otimes B\)};
			      \node (TR) at (4,0) {\(A \otimes B\)};
			      \node (BM) at (2,2) {\(B \otimes A\)};
			      \draw [double distance=2pt] (TL) to node[below]{\(\id_{A\otimes B}\)} (TR);
			      \draw [->] (TL) to node[left]{\(\sym_{A,B}\)} (BM);
			      \draw [->] (BM) to node[right]{\(\sym_{B, A}\)} (TR);
		      }

		      \tikz{
			      \node (TL) at (0,2) {\((A \otimes B) \otimes C\)};
			      \node (TR) at (5,2) {\((B \otimes A) \otimes C\)};

			      \node (ML) at (0,0) {\(A \otimes (B \otimes C)\)};
			      \node (MR) at (5,0) {\(B \otimes (A \otimes C)\)};

			      \node (BL) at (0,-2) {\((B \otimes C) \otimes A\)};
			      \node (BR) at (5,-2) {\(B \otimes (C \otimes A)\)};

			      \draw [->] (TL) to node[above]{\(\sym_{A,B} \otimes \id_C\)} (TR);
			      \draw [->] (TL) to node[left]{\(\assoc_{A, B, C}\)} (ML);
			      \draw [->] (ML) to node[left]{\(\sym_{A, B \otimes C} \)} (BL);
			      \draw [->] (BL) to node[below]{\(\assoc_{B \otimes C, A}\)} (BR);
			      \draw [->] (TR) to node[right]{\(\assoc_{B, A, C}\)} (MR);
			      \draw [->] (MR) to node[right]{\(\id_{B} \otimes \sym_{A, C}\)} (BR);
		      }

		      }

		\item[closed monoidal]
		      if it is monoidal and, for all \(B\), the functor \(\otimes_B : \_ \to \_ \otimes B\)\footnote{This functor can be thought of as a \enquote{partially applied bifunctor}.} has a right adjoint (\autoref{def:prop-functors}) \(\Rightarrow_B : \_ \to B \Rightarrow \_\).
		\item[Cartesian monoidal]
		      if its monoidal structure is given by the (binary cartesian) product: the bifunctor \(\otimes\) is the product, the neutral object is the terminal object \(\one\), and, for every \(A\), \(B\) and \(C\),
		      \begin{itemize}
			      \item \(\assoc_{A,B,C} : (A \times B) \times C \to A \times (B \times C)\), the associator, is \(\langle \pi_1 \circ \pi_1, \pi_2 \times \id_C \rangle\),
			      \item \(\lunitor_{A} : \one \times A \to A\), the left unitor, is \(\pi_2\),
			      \item \(\runitor_A : A \times \one \to A\), the right unitor, is \(\pi_1\).
		      \end{itemize}

		      Note that every cartesian monoidal category is symmetric monoidal, with \(\sym_{A, B} = \pi_2 \times \pi_1 : A \times B \to B \times A\).
		\item[Cartesian closed] if it has a terminal object and every object is exponentiating, or, equivalently, if it has a terminal object, and every pair of objects have an exponent and a product.
		\item[discrete] if the only morphisms are the identities.
		\item[a preorder category] if there is at most one morphism between any two objects.
		\item[well-pointed] if it has a terminal object \(\one\) and for all \(f_1, f_2 : A \to B\) such that \(f_1 \neq f_2\), there exists \(p : \one \to A\), a \emph{global object} (or \emph{point}) such that \(f_1 \circ p \neq f_2 \circ p\).
		\item[pointed] if it has a zero object.
	\end{description}
\end{definition}

We refer to \autoref{sec:cheat-cartesian} for a series of equalities concerning the binary product, the exponents and the associator for the product.

We will often omit the subscripts on the natural transformations, objects and morphisms above when they can be infered from context.
We also work up to associativity most of the time. %

\begin{definition}[Properties of functors]
	\label{def:prop-functors}
	Given two functors \(F : \ccat{D} \to \ccat{C}\) and \(G : \ccat{C} \to \ccat{D}\),
	\begin{description}[style=unboxed]
		\item[\(F\) is a left adjoint to \(G\) (and \(G\) is a right adjoint to \(F\)) {\cite[492]{Maclane1971}}] if for all \(D \) in \( \ccat{D}\), \(C\) in \( \ccat{C}\), \(\Hom_{\ccat{C}}(FD, C) \cong \Hom_{\ccat{D}} (D, GC)\) are natural in the variables \(C\) and \(D\), that is, \(F\) and \(G\) are equiped with natural transformations \(\eta : \id_{\ccat{C}} \natt F \circ G\) and \(\varepsilon : G \circ F \natt \id_{\ccat{D}}\) such that for all \(D\) in \( \ccat{D}\), \(C\) in \( \ccat{C}\), the following commutes:

		      {\centering

		      \tikz{
			      \node (GC1) at (0,2) {\(GC\)};
			      \node (GFGC) at (3,2) {\(GFGC\)};
			      \node (GC2) at (3,0) {\(GC\)};
			      \draw [->] (GC1) to node[above]{\(G (\eta_C)\)} (GFGC);
			      \draw [->] (GFGC) to node[right]{\(\varepsilon_{GC}\)} (GC2);
			      \draw [double distance=2pt] (GC1) to node[left = 0.3cm]{\(\id_{GC}\)} (GC2);
		      }
		      \tikz{
			      \node (FD1) at (0,2) {\(FD\)};
			      \node (GFFD) at (3,2) {\(FGFD\)};
			      \node (FD2) at (3,0) {\(FD\)};
			      \draw [->] (FD1) to node[above]{\( \eta_{FD}\)} (GFFD);
			      \draw [->] (GFFD) to node[right]{\(F(\varepsilon_{D})\)} (FD2);
			      \draw [double distance=2pt] (FD1) to node[left = 0.3cm]{\(\id_{FD}\)} (FD2);
		      }

		      }
		      We write \(F \dashv G\).
		\item[\(F\) is full] if for every \(D, E\) in \( \ccat{D}\), for all \(g : FD \to FE\), there exists \(h : D \to E\) such that \(g = F h\).
		\item[\(F\) is faithfull (or an embedding)] if for every \(D\), \(E\) in \( \ccat{D}\), for all \(f_1, f_2 : D \to E\), \(F f_1 = F f_2 \implies f_1 = f_2\).
		\item[\(F\) is fully faithful] if it is full and faithful.
		\item[\(F\) is contravariant (or a co-functor)] if for all \(f : A \to B\) in \(\ccat{D}\), \(Ff : Fb \to Fa\).
		\item[\(F\) is a bifunctor (or a binary functor)] if \(\ccat{D}\) is the product of two categories.
		\item[\(F\) is an endofunctor] if \(\ccat{D} = \ccat{C}\), and we write \(F^n\) for the application of \(F\) \(n\) times.
		\item[\(F\) is a (left) strong functor{~\cite[Definition 2.6.7]{Jacobs1999}}] if \(F\) is an endofunctor on a monoidal category \(\ccat{C}\) endowed with a \emph{(tensorial) (left) strength} \(\lst\), a natural transformation \(\lst_{A, B} : A \otimes FB \to F(A \otimes B)\) such that the following commutes:

		      {\centering

		      \tikz%
		      {
			      \node (TL) at (0, 2) {\((A \otimes B) \otimes FC\)};
			      \node (TR) at (0,-2) {\( F((A \otimes B) \otimes C)\)};
			      \node (BL) at (8, 2) {\( A \otimes (B \otimes FC) \)};
			      \node (BM) at (8, 0) {\(A \otimes F (B \otimes C)\)};
			      \node (BR) at (8, -2) {\(F(A \otimes (B \otimes C ))\)};
			      \draw [->] (TL) to node[left]{\(\lst_{A \otimes B, C }\)} (TR);
			      \draw [->] (TL) to node[above]{\(\assoc_{A, B, FC} \)} (BL);
			      \draw [->] (BL) to node[right]{\(\id_A \otimes \lst_{B, C}\)} (BM);
			      \draw [->] (BM) to node[right]{\(\lst_{A, B \otimes C}\)} (BR);
			      \draw [->] (TR) to node[above]{\(F(\assoc_{A, B, C})\)} (BR);
		      }

		      \tikz{
			      \node (TL) at (0,2) {\(I \otimes F A\)};
			      \node (TR) at (4,2) {\(F (I \otimes A) \)};
			      \node (BM) at (2,0) {\(FA\)};
			      \draw [->] (TL) to node[above]{\(\lst_{I,A}\)} (TR);
			      \draw [->] (TL) to node[left]{\(\lunitor_{FA}\)} (BM);
			      \draw [->] (TR) to node[right = 0.2cm]{\(F(\lunitor_{A})\)} (BM);
		      }

		      }
		\item[\(F\) is a right strong functor] if \(F\) is an endofunctor on a monoidal category \(\ccat{C}\) endowed with a \emph{(tensorial) right strength} \(\rst\), a natural transformation \(\rst_{A, B} : FA \otimes B \to F(A \otimes B)\) such that the following commutes:

		      {\centering

		      \tikz%
		      {
			      \node (TL) at (0, 2) {\(FA \otimes (B \otimes C)\)};
			      \node (TR) at (0,-2) {\( F(A \otimes (B \otimes C))\)};
			      \node (BL) at (6, 2) {\( (FA \otimes B) \otimes C \)};
			      \node (BM) at (6, 0) {\(F(A \otimes B) \otimes C\)};
			      \node (BR) at (6, -2) {\(F((A \otimes B) \otimes C)\)};
			      \draw [->] (TL) to node[left]{\(\rst_{A, B \otimes C }\)} (TR);
			      \draw [->] (TL) to node[above]{\(\assoc^{-1}_{FA, B, C} \)} (BL);
			      \draw [->] (BL) to node[right]{\(\rst_{A, B} \otimes \id_C\)} (BM);
			      \draw [->] (BM) to node[right]{\(\rst_{A \otimes B, C}\)} (BR);
			      \draw [->] (TR) to node[above]{\(F(\assoc^{-1}_{A, B, C})\)} (BR);
		      }

		      \tikz{
			      \node (TL) at (0,2) {\(F A \otimes I\)};
			      \node (TR) at (4,2) {\(F (A \otimes I) \)};
			      \node (BM) at (2,0) {\(FA\)};
			      \draw [->] (TL) to node[above]{\(\rst_{A, I}\)} (TR);
			      \draw [->] (TL) to node[left]{\(\runitor_{FA}\)} (BM);
			      \draw [->] (TR) to node[right = 0.2cm]{\(F(\runitor_{A})\)} (BM);
		      }

		      }
	\end{description}
\end{definition}
Note that a fully faithful functor may not be an isomorphism of categories, and that a \enquote{strong} functor usually refers to a \emph{left} strong functor.

\section{Constructions over Categories and Functors}

\begin{definition}[Constructions over categories]
	\label{def:construction_cat}
	Given \(\ccat{B}\) and \(\ccat{C}\) two categories, and \(B\) an object of \(\ccat{B}\),
	\begin{description}[style=unboxed]
		\item[\(\ccat{B}\) is a subcategory of \(\ccat{C}\)] if \(\ccat{B}\) is a category, whose objects are a subcollection of objects of \(\ccat{C}\), and whose morphisms are a subcollection of morphisms of \(\ccat{C}\).
		\item[The arrow category \(\ccat{B}^{\to}\)~{\cite[28]{Jacobs1999}} (or the category of arrows of \(\ccat{B}\) \(\ccat{B}^2\)~{\cite[40--41]{Maclane1971}})] has
		      \begin{description}[style=unboxed]
			      \item[for objects] morphisms of \(\ccat{B}\),
			      \item[for morphisms] couples \((u, g) : f_1 \to f_2\) of morphisms of \(\ccat{B}\) %
			            such that the following commutes:

			            {\centering

			            \tikz{\node (b1) at (0,0) {\(B_1\)};
				            \node (b'1) at (0,-2) {\(B'_1\)};
				            \node (b2) at (2,0) {\(B_2\)};
				            \node (b'2) at (2,-2) {\(B'_2\)};
				            \draw [->] (b1) to node[above]{\(g\)} (b2);
				            \draw [->] (b1) to node[left]{\(f_1\)} (b'1);
				            \draw [->] (b'1) to node[above]{\(u\)} (b'2);
				            \draw [->] (b2) to node[right]{\(f_2\)} (b'2);
			            }

			            }
		      \end{description}
		\item[The slice category \(\slic{B}{B}\)~{\cite[28]{Jacobs1999}} (or the category of object over \(B\), or over category)] is the subcategory of \(\ccat{B}^{\to}\), which have
		      \begin{description}[style=unboxed]
			      \item[for objects] morphisms of \(\ccat{B}\) whose codomain is \(B\),
			      \item[for morphisms] the morphisms of \(\ccat{B}^{\to}\) whose first component is \(\id_B\).
		      \end{description}
		\item[The co-slice category \(\coslic{B}{B}\)~{\cite[28]{Jacobs1991}} (or the under category, \((B \downarrow \ccat{B})\) or \((B / \ccat{B})\))] is the subcategory of \(\ccat{B}^{\to}\), which have
		      \begin{description}[style=unboxed]
			      \item[for objects] morphisms of \(\ccat{B}\) whose domain is \(B\),
			      \item[for morphisms] the morphisms of \(\ccat{B}^{\to}\) whose second component is \(\id_B\).
		      \end{description}
		\item[The opposite category \(\ccat{B}^{\op}\)~{\cite[33]{Maclane1971}}] has
		      \begin{description}[style=unboxed]
			      \item[for objects] objects of \(\ccat{B}\),
			      \item[for morphisms] \(f^{\op} : B \to A\) for each morphism \(f : A \to B\) in \(\ccat{B}\).
		      \end{description}
		\item[The functor category \(\ccat{B}^{\ccat{C}}\)~{\cite[40]{Maclane1971}} (or\(\text{Funct}(\ccat{C}, \ccat{B})\))] has
		      \begin{description}[style=unboxed]
			      \item[for objects] functors \(F : \ccat{C} \to \ccat{B}\),
			      \item[for morphisms] natural transformations between functors from \(\ccat{C}\) to \(\ccat{B}\).
		      \end{description}

	\end{description}
\end{definition}

\begin{definition}[Constructions over functors]
	Given \(\ccat{A}\), \(\ccat{B}\) and \(\ccat{C}\) three categories, \(F : \ccat{B} \to \ccat{C}\) and \(G : \ccat{A} \to \ccat{C}\) two functors,
	\begin{description}[style=unboxed]
		\item[the opposite of \(F\)] is the unique functor \(F^{\op} : \ccat{B}^{\op} \to \ccat{C}^{\op}\).
		\item[the comma category \((G \downarrow F)\){~\cite[45--46]{Maclane1971}}] has
		      \begin{description}[style=unboxed]
			      \item[for objects] triples \((A, B, f)\) such that \(A \) in \( \ccat{A}\), \(B\) in \( \ccat{B}\) and \(f : GA \to FB\) is a morphism in \(\ccat{C}\).
			      \item[for morphisms] pairs \((g, h) : (A_1, A_2, f_1) \to (B_1, B_2, f_2)\) of morphisms in \(\ccat{A}\) and \(\ccat{B}\) respectively, such that \(g: A_1 \to A_2\), \(h : B_1 \to B_2\)) the following commutes:

			            {\centering

			            \tikz{\node (ga) at (0,0) {\(GA_1\)};
				            \node (fb) at (0,-2) {\(FB_1\)};
				            \node (ga') at (2,0) {\(GA_2\)};
				            \node (fb') at (2,-2) {\(FB_2\)};
				            \draw [->] (ga) to node[above]{\(Gg\)} (ga');
				            \draw [->] (ga) to node[left]{\(f_1\)} (fb);
				            \draw [->] (fb) to node[above]{\(Fh\)} (fb');
				            \draw [->] (ga') to node[right]{\(f_2\)} (fb');
			            }

			            }

		      \end{description}
	\end{description}
\end{definition}

\begin{remark}%
	[Comma category as a general construction]~
	\begin{itemize}
		\item If \(\ccat{A} = \ccat{C}\), \(G = \id_{\ccat{C}}\) and \(\ccat{B} = \ccat{1}\), then if \(F \ccat{1} = c\) for \(c \) in \( \ccat{C}\), \((G \downarrow F)\) is precisely \(\slic{C}{c}\) the slice category over \(c\).
		\item If \(\ccat{B} = \ccat{C}\), \(F = \id_{\ccat{C}}\) and \(\ccat{A} = \ccat{1}\), then if \(G \ccat{1} = c\) for \(c \) in \( \ccat{C}\), \((G \downarrow F)\) is precisely \(\coslic{c}{C} \) the coslice category over \(c\).
		\item If \(F\) and \(G\) are the identity functor of \(\ccat{C}\), then \((F \downarrow G)\) is the arrow category \(\ccat{C}^{\to}\).
		\item If \(F\) and \(G\) are both functor with domain \(\ccat{1}\), and \(F \ccat{1} = A\), \(G \ccat{1} = B\), then \((F \downarrow G)\) is the discrete category whose objects are morphisms from \(A\) to \(B\).
	\end{itemize}
\end{remark}

\chapter{On Fibrations}

\begin{definition}[Fibration {\cites[Definition 1.2.]{Hofstra2008}[49]{Jacobs1999}}]
	\label{def:fibration}
	The functor \(U : \ccat{E} \to \ccat{B}\) is
	\begin{description}[style=unboxed]
		\item[a fibration]if for every \(Y \) in \( \ccat{E}\) and \(u : I \to UY\) in \(\ccat{B}\), there is a cartesian morphism \(f: X \to Y\) in \(\ccat{E}\) above \(u\).
		\item[an opfibration (or cofibration)~{\cite[Section 9.1]{Jacobs1991}}] if \(U^{\op} : \ccat{E}^{\op} \to \ccat{B}^{\op}\) is a fibration.
		\item[a bifibraction] if it is a fibration and an opfibration.
		\item[a cloven (op)fibration] if it is an (op)fibration and has a \emph{cleavage}, i.e.\ a choice of (op)cartesian liftings.
	\end{description}
\end{definition}

\begin{definition}[Fibre]
	If \(U : \ccat{E} \to \ccat{B}\) is a fibration and \(X\) is an object in \(\ccat{B}\), then we write \(\ccat{E}_X\) and call \emph{the fibre over \(X\)} the category whose
	\begin{description}[style=unboxed]
		\item[objects] are the objects \(Y\) in \(\ccat{E}\) such that \(UY = X\),
		\item[morphisms] are the morphisms \(f\) in \(\ccat{E}\) such that \(Uf = \id_X\).
	\end{description}
\end{definition}

\begin{definition}[Re-indexing (or substitution, relabeling, inverse image, transition) functor~{\cites[268]{Johnstone2002}[48--49]{Jacobs1999}}]
	\label{def:re-indexing}
	Given \(U : \ccat{E} \to \ccat{B}\) a cloven fibration, for all \(f : X \to UP\) in \(\ccat{B}\), we define the \emph{re-indexing functor} \(\reind{f} : \ccat{E}_{UP} \to \ccat{E}_{X}\) as
	\begin{description}[style=unboxed]
		\item[on objects] \(\reind{f}P\) is the domain of \(f^{\S}_{P}\),
		\item[on morphisms] for \(k : P \to P'\), \(\reind{f}k\) is given by cartesianity of \(f^{\S}_{P'}\) along \(k \circ f^{\S}_P\):
	\end{description}

	{\centering

	\begin{tikzpicture}
		\node (f*P) at (0, 6) {\(\reind{f}P\)};
		\node (f*P') at (1, 4) {\(\reind{f}P'\)};
		\node (P) at (4, 6) {\(P\)};
		\node (P') at (5, 4) {\(P'\)};
		\draw[->] (P) to node[right]{\(k\)} (P');
		\draw[->, dashed] (f*P) to node[left]{\(\reind{f}k\)} (f*P');
		\draw[->] (f*P') to node[above]{\(f^{\S}_{P'}\)} (P');
		\draw[->] (f*P) to node[above]{\(f^{\S}_{P}\)} (P);
	\end{tikzpicture}

	}
\end{definition}

\begin{definition}[Opreindexing (or extension, sem) functor]
	\label{def:opre-indexing}
	Let \(U : \ccat{E}\to \ccat{B}\) be a cloven opfibration, \(f : UP \to Y\) be a morphism in \(\ccat{B}\).
	We define the \emph{opreindexing} functor \(\opreind{f} : \ccat{E}_{UP} \to \ccat{E}_{Y}\) as
	\begin{description}[style=unboxed]
		\item[on objects] \(\opreind{f} P\) is the codomain of \(f_{\S}^{P}\),
		\item[on morphisms] for \(k : P \to P'\), \(\opreind{f} k\) is given by opcartesianity of \(f_{\S}^{P}\) along \(f_{\S}^{P'} \circ k\).
	\end{description}

	{\centering

	\begin{tikzpicture}
		\node (sigmafP) at (4, 6) {\(\opreind{f}(P)\)};
		\node (sigmafP') at (5, 4) {\(\opreind{f}(P')\)};
		\node (P) at (0, 6) {\(P\)};
		\node (P') at (1, 4) {\(P'\)};
		\draw[->] (P) to node[right]{\(k\)} (P');
		\draw[->, dashed] (sigmafP) to node[left]{\(\opreind{f} k\)} (sigmafP');
		\draw[->] (P') to node[above]{\(f_{\S}^{P'}\)} (sigmafP');
		\draw[->] (P) to node[above]{\(f_{\S}^{P}\)} (sigmafP);
	\end{tikzpicture}

	}
\end{definition}

\begin{definition}[Properties of fibres~{\cite[27]{Tews2002}}]
	Let \(U : \ccat{E}\to \ccat{B}\) be a cloven opfibration, \(J\) be an object in \(\ccat{B}\), and \(A, B \) be in \( \ccat{E}_J\).
	The fibre \(\ccat{E}_{J}\) has \emph{fibred product} if
	\begin{itemize}
		\item the fibre \(\ccat{E}_{J}\) has product, denoted \(\times_{\ccat{E}_J}\) below
		\item \(\forall u : I \to J\) in \(\ccat{B}\), \(u^*(A \times_{\ccat{E}_J} B) \cong (u^* A) \times_{\ccat{E}_I} (u^* B)\).
	\end{itemize}
	it has \emph{fibred exponent} if
	\begin{itemize}
		\item the fibre \(\ccat{E}_{J}\) has exponent, denoted \(\expo_{\ccat{E}_J}\) below,
		\item \(\forall u : I \to J\) in \(\ccat{B}\), \(u^*(B \expo_{\ccat{E}_I} A) \cong (u^* B) \expo_{\ccat{E_J}}(u^*A)\).
	\end{itemize}
\end{definition}

\begin{definition}[Generic objects~{\cite[Definition 5.2.8]{Jacobs1999}}]
	Let \(U : \ccat{E} \to \ccat{B}\) be a fibration, an object \(X \) in \( \ccat{E}\) is
	\begin{description}[style=unboxed]
		\item[weak generic (or generic~{\cites[Definition 1.2.9]{Jacobs1991}[47--48]{Hermida1993}})] if for all \(Y\) in \( \ccat{E}\), there exists \(f:Y \to X\) and \(f\) is cartesian.
		\item[generic] if for all \(Y\) in \( \ccat{E}\), there exists a unique \(u: UY \to UX\) and there exists \(f : Y \to X\) cartesian over \(u\).
		\item[strong generic] if for all \(Y\) in \( \ccat{E}\), there exists a unique \(f : Y \to X\) and \(f\) is cartesian.
		\item[split generic {\cites[Definition 5.2.1]{Jacobs1999}[Definition 1.2.11]{Jacobs1991}}] if \(U\) is a split fibration, and there exists a collection of isomorphisms
		      \[\theta_I : \Hom_{\ccat{B}}(I, UX) \cong \Obj(\ccat{E}_I)\]
		      with \(\theta_J (u \circ v) = v^* (\theta_I u)\) for \(v : J \to I\).
	\end{description}
	The image \(U X\) in \(\ccat{B}\) is written \(\Omega\).
\end{definition}

\begin{definition}[Properties of fibrations]
	A fibration \(U : \ccat{E} \to \ccat{B}\) has
	\begin{description}[style=unboxed]
		\item[product {\cite[97]{Jacobs1999}}] if \(\ccat{B}\) has pullback and all re-indexing functor \(u^*\) has a right adjoint \(\Pi_u\) that respects in some way the Belk-Chevalley condition.
		\item[simple product {\cite[94]{Jacobs1999}}] if \(\ccat{B}\) has product and all substitution functors along the cartesian product projections \(\pi_{I, J} : I \times J \to I\) have a right adjoint \(\Pi_{I,J}\) that respects in some way the Belk-Chevalley condition.
		\item[simple \(\Omega\)-product] if \(\ccat{B}\) has product and simple product for cartesian projections of products with \(\Omega\).
		\item[exponent {\cite[Definition 3.9, p.~179]{Jacobs1993}}] if it has cartesian product and some functor has a fibred right adjoint.
		\item[fibred product~{\cite[42]{Hermida1993}}] if every fibre has fibred product.
		\item[a fibred terminal object~{\cite[42--43]{Hermida1993}}] if each fibre \(\ccat{E}_X\) has a terminal object \(\one_{\ccat{E}_X}\), and if reindexing preserves terminal object: \(\forall X, Y, f: U X \to U Y\), \(f^*\one_{\ccat{E}_Y} = \one_{\ccat{E}_X}\).
	\end{description}
	it is
	\begin{description}[style=unboxed]
		\item[a split fibration~{\cite[49--50]{Jacobs1999}}] if it is cloven, and additionally, \(\id^* = \id\) and \((v \circ u)^* = u^* \circ v^*\),
		\item[a polymorphic fibration {\cite[p. 471]{Jacobs1999}}] if it has a generic object, fibred finite product, and finite products in \(\ccat{B}\).
		\item [a partial order (or preordered) fibration~{\cites[233]{Katsumata2013}[43]{Jacobs1999}}]
		      if every fibre is a preorder category.
	\end{description}
\end{definition}

\begin{lemma}
	A fibration is faithful if and only if it is a partial order fibration.
\end{lemma}

\begin{proof}
	Let \(U : \ccat{E} \to \ccat{B}\) be a fibration.
	\begin{description}[style=unboxed]
		\item[\(\Rightarrow\)] Let \(A \) in \( \ccat{B}\) and \(f, g \) in \( \ccat{E}_A\), both are vertical, i.e.\ \(U f = \id_A = U g\), but as \(U\) is faithful, \(f = g\), i.e.\ \(\ccat{E}_A\) is a preorder.
		\item[\(\Leftarrow\)] Assume \(f, g : P \to Q\) and \(Uf = Ug\).
		      As every morphism in \(\ccat{E}\) can be factorised as the composition of a vertical morphism and a cartesian lifting~\cite[1.1.3, p.29]{Jacobs1999}, i.e.\ \(f = h \circ f'\) and \(g = u \circ g'\).
		      Both \(h\) and \(u\) are endomorphisms in \(\ccat{E}_P\), herce, by partial ordering, \(h = u\).
		      Moreover, a cartesian morphism is unique up to isomorphism in a fibre~\cite[Proposition 1.1.4]{Jacobs1999}, so that \(f' = g'\), and \(f = g\).\qedhere
	\end{description}
\end{proof}

\chapter{On Slice Categories}
\label{sec:slice}

\section{Preliminaries on Slices}

We start by coming back to the definition of slice category (\autoref{def:construction_cat}) and introduce proper notations for it.

\begin{definition}[Slice category]
	\label{slice-def}
	Let \(\ccat{C}\) be a category, and \(A\) be an object of \(\ccat{C}\).
	The \emph{slice category \(\slic{C}{A}\)} is the category whose
	\begin{description}[style=unboxed]
		\item[objects] are pairs \((X, f_X)\) such that \(X\) is an object of \(\ccat{C}\) and \(f_X : X \to A\) is a morphism of \(\ccat{C}\),
		\item[morphisms] \(\is{h} : (X, f_X) \to (Y, f_Y)\) are morphism \(h : X \to Y\) in \(\ccat{C}\) such that \(f_Y \circ h = f_X\) in \(\ccat{C}\):

		      {\centering

		      \tikz{
			      \node (x) at (0, 0) {\(X\)};
			      \node (y) at (2, 0) {\(Y\)};
			      \node (a) at (1, -1) {\(A\)};
			      \draw [->] (x) to node[above]{\(h\)} (y);
			      \draw [->] (x) to node[left=0.2]{\(f_X\)} (a);
			      \draw [->] (y) to node[right=0.2]{\(f_Y\)} (a);
		      }

		      }
		\item[identity] on \((X, f_X)\)is \(\id_X\),
		\item[composition] is defined as \(\is{h} \circ_{\slic{C}{A}} \is{g} = \is{h \circ_{\ccat{C}} g}\).
	\end{description}
\end{definition}

For the following definition, we will use the definition and notation relative to the pullback introduced in \autoref{def-prop-cat}.

\begin{definition}[Pullback (or change-of-base) functor~{\cites[13.4.1]{Barr2012}[Proposition 5.10]{Awodey2010}}]
	\label{pullback-def}
	Let \(\ccat{C}\) be a category with all pullbacks, \(f : A \to B\) be a morphism in \(\ccat{C}\), we define the the \emph{pullback functor} \(\pull{f} : \slic{C}{B} \to \slic{C}{A}\) as
	\begin{description}[style=unboxed]
		\item[on objects] \(\pull{f}(C, f_C)\) is the pair \((fC, f^*f_C)\) given by the pullback of \(f\) along \(f_C\):

		      {\centering

		      \tikz{
			      \node (A) at (0, 0) {\(A\)};
			      \node (B) at (2, 0) {\(B\)};
			      \node (C) at (2, 1) {\(C\)};
			      \node (fC) at (0, 1) {\(fC\)};
			      \draw [->] (C) to node[right]{\(f_C\)} (B);
			      \draw [->] (A) to node[below]{\(f\)} (B);
			      \draw [->] (fC) to node[left]{\(f^* f_C\)} (A);
			      \draw [->] (fC) to node[above]{\(f_C^* f\)} (C);
		      }

		      }

		\item[on morphisms] \(\pull{f}\is{m}\), for \(\is{m} : (C, f_C) \to (D, f_D)\) a morphism in \(\slic{C}{B}\), is the unique morphism between \(\pull{f}(C, f_C)\) and \(\pull{f}(D, f_D)\) given by taking \((f^*C, f^*f_C, m \circ f_C^*f)\) as the \enquote{alternative} pullback of \(f\) and \(f_D\).
	\end{description}
	More precisely, we have:

	
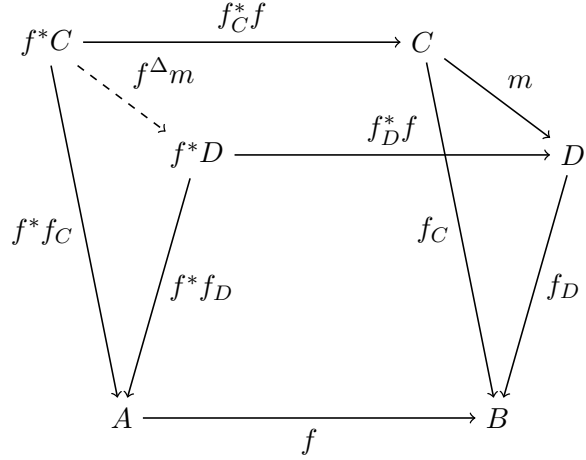
{\centering

	\begin{tikzpicture}
		\node (fc) at (-1, 5) {\(f^*C\)};
		\node (fd) at (1, 3.5) {\(f^*D\)};
		\node (a) at (0, 0) {\(A\)};
		\node (c) at (4, 5) {\(C\)};
		\node (d) at (6, 3.5) {\(D\)};
		\node (b) at (5, 0) {\(B\)};
		\draw [->] (fc) to node[left]{\(f^*f_C\)} (a);
		\draw [->] (fd) to node[right]{\(f^*f_D\)} (a);
		\draw [->] (c) to node[left]{\(f_C\)} (b);
		\draw [->] (d) to node[right]{\(f_D\)} (b);
		\draw [->] (c) to node[above right]{\(m\)} (d);
		\draw [->] (a) to node[below]{\(f\)} (b);
		\draw [->] (fc) to node[above]{\(f_C^*f\)} (c);
		\draw [->] (fd) to node[above]{\(f_D^*f\)} (d);
		\draw [->, dashed] (fc) to node[above right]{\(\pull{f}m\)} (fd);
	\end{tikzpicture}

	}
	Then, since \(f^*f_D\) is the pullback of \(f\) and \(f_D\), and since \(f \circ f^*f_C = f_D \circ (m \circ f_C^*f)\) (because \(f_D \circ m = f_C\), since \(m\) is a morphism in \(\slic{C}{B}\), and \(f \circ f^*f_C = f_C \circ f_C^*f\), by construction), there exists a unique \(\pull{f}m\) such that \(f^*f_C = f^*f_D\circ \pull{f}m\), and so \(\pull{f}m\) is a morphism in \(\slic{C}{A}\).
\end{definition}

In the future, we will prefer the notation \(A \times_{\ccat{C}}^B D\) over \(f^*D\).

\begin{remark}
	This pullback functor \(\pull{f} : \slic{C}{B} \to \slic{C}{A} \) is not to be confused with the reindexing functor \(f^* : \ccat{E}_{UP} \to \ccat{E}_{X}\) defined thanks to a cloven fibration \(U: \ccat{E}\to \ccat{B}\) in \autoref{def:re-indexing}, even if they are sometimes both denoted with \(f^*\)~\cite%
	[Example~7.29]{Awodey2011}.%
	The fact that the same notation is used for both probably comes from the fact that if \(U\) %
	is the codomain functor \(\codom : \ccat{B}^{\to} \to \ccat{B}\)%
	, then \(\ccat{C}_Z \cong \slic{C}{Z}\) for all \(Z\) in \(\ccat{C}\)~{\cite[28, Ex. 1.4.2]{Jacobs1999}}, and the assimilation is grounded.
\end{remark}

\begin{definition}[Composition functor~{\cite[13.4.2]{Barr2012}}]
	\label{composition-def}
	Let \(f : A \to B\) be a morphism in \(\ccat{C}\),
	we define \emph{the composition functor} \(\comp{f} : \slic{C}{A} \to \slic{C}{B}\) to be
	\begin{description}[style=unboxed]
		\item[on objects] \(\comp{f}(C, f_C)\) is \((C, f \circ f_C)\),
		\item[on morphisms] \(\comp{f}\is{k}\), for \(\is{k} : (C, f_C) \to (D, f_D)\) a morphism in \(\slic{C}{A}\), is \(\is{k}\) itself:

		      
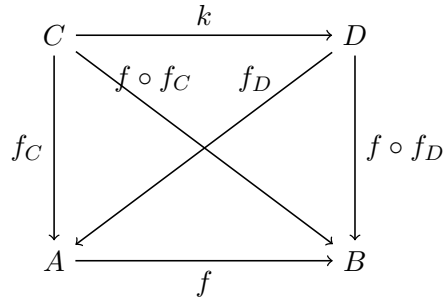
{\centering

		      \begin{tikzpicture}
			      \node (a) at (0, 0) {\(A\)};
			      \node (b) at (4, 0) {\(B\)};
			      \node (c) at (0, 3) {\(C\)};
			      \node (d) at (4, 3) {\(D\)};
			      \draw [->] (a) to node[below]{\(f\)} (b);
			      \draw [->] (c) to node[left]{\(f_C\)} (a);
			      \draw [->] (d) to node[above=0.1, pos=0.3]{\(f_D\)} (a);
			      \draw [->] (d) to node[right]{\(f \circ f_D \)} (b);
			      \draw [->] (c) to node[above=0.1, pos=0.3]{\(f \circ f_C\)} (b);
			      \draw [->] (c) to node[above]{\(k\)} (d);
		      \end{tikzpicture}

		      }

		      We can easily make sure that \(\is{k}\) is a morphism in \(\slic{C}{B}\): \(f \circ f_D \circ k = f \circ f_C\) holds since \(k\) is a morphism in \(\slic{C}{A}\).
	\end{description}
\end{definition}

Given \(f : B \to A\), \(\pull{f}\) and \(\comp{f}\) are actually adjoints, and it can be the case that \(\pull{f}\) also has a right adjoint:

\begin{definition}[Adjoints of \(\pull{f}\)~{\cites[Definition 9.19]{Awodey2011}[Corollary A.1.5.3]{Johnstone2002}}]
	\label{def:adjunction-to-f}
	Let \(f : B \to A\), if \(\pull{f}\) has a right adjoint, then we write it \(\pullra{f}\), say that \(f\) is \emph{exponentiable}, and we have:

	\[ \comp{f} \dashv \pull{f} \dashv \pullra{f}\]

	I.e.,

	{\centering

			\tikz{
				\node (CA) at (0,0) {\(\slic{C}{A}\)};
				\node (CB) at (4,0) {\(\slic{C}{B}\)};
				\draw [->] (CA) to node[above]{\(\pull{f}\)} (CB);
				\draw [->] (CB) to [bend left] node[below]{\(\comp{f}\)} (CA);
				\draw [->] (CB) to [bend right] node[above]{\(\pullra{f}\)} (CA);
			}

		}
	In particular, for the adjunction \(\comp{f} \dashv \pull{f}\), we have the unit \(\unit_{\comp{f}} : \id_{\slic{C}{A}} \natt \pull{f}\comp{f}\) and the counit \(\counit_{\comp{f}} : \comp{f}\pull{f} \natt \id_{\slic{C}{B}} \) such that for all \((C, f_C)\) in \(\slic{C}{B}\), \((D, f_D)\) in \(\slic{C}{A}\),

	\begin{align}
		\forall \is{k} : (C, f_C) \to \pull{f}(D, f_D), & \exists ! \is{l} : \comp{f} (C, f_C) \to (D, f_D) \notag                                        \\
		                                                & \text{ s.t. } \is{k} = \pull{f}\is{l} \circ (\unit_{\comp{f}})_{(C, f_C)} \label{unit-comp}     \\
		\forall \is{k} : \comp{f}(C, f_C) \to (D, f_D), & \exists ! \is{l} : (C, f_C) \to \pull{f}(D, f_D)\notag                                          \\
		                                                & \text{ s.t. } \is{k} = (\counit_{\comp{f}})_{(D, f_D)} \circ \comp{f}\is{l} \label{counit-comp}
	\end{align}

	And, for the adjunction \(\pull{f} \dashv \pullra{f}\), the unit \(\unit_{\pullra{f}} : \id_{\slic{C}{B}} \natt \pullra{f}\pull{f}\) and the counit \(\counit_{\pullra{f}} : \pull{f}\pullra{f} \natt \id_{\slic{C}{A}} \) are such that
	\begin{align}
		\forall \is{k} : (D, f_D) \to \pullra{f}(C, f_C), & \exists ! l : \pull{f} (D, f_D) \to (C, f_C) \notag                                                 \\
		                                                  &
		\text{ s.t. } \is{k} = \pullra{f}\is{l} \circ (\unit_{\pullra{f}})_{(D, f_D)} \label{unit-pullra}                                                       \\
		\forall \is{k} : \pull{f} (D, f_D) \to (C, f_C),  & \exists ! l : (D, f_D) \to \pullra{f}(C, f_C)\notag                                                 \\
		                                                  & \text{ s.t. } \is{k} = (\counit_{\pullra{f}})_{(C, f_C)} \circ \pull{f}\is{l} \label{counit-pullra}
	\end{align}
\end{definition}

In the following (\autoref{sub:slice-term}, \autoref{sub:slice-product} and \autoref{sub:slice-exponent}), we'll prove that, under certain conditions, the slice category \(\slic{C}{A}\) can be endowed with a cartesian structure~\cite[Proposition 9.20]{Awodey2010}, with \((A, \id_A)\) being the terminal object, and, for \((X_1, f_{X_1})\) and \((X_2, f_{X_2})\) two objects,

\begin{align}
	(X_1, f_{X_1}) \times (X_2, f_{X_2})      & =_{\text{def}} \comp{f_{X_1}}(\pull{f_{X_1}}(X_2, f_{X_2})) \label{eq-prod-1}   \\
	\shortintertext{or, equivalently}
	(X_1, f_{X_1}) \times (X_2, f_{X_2})      & =_{\text{def}} \comp{f_{X_2}}(\pull{f_{X_2}}(X_1, f_{X_1})) ) \label{eq-prod-2} \\
	(X_1, f_{X_1}) \Rightarrow (X_2, f_{X_2}) & =_{\text{def}} \pullra{f_{X_1}}(\pull{f_{X_1}}(X_2, f_{X_2})) ) \label{eq-expo}
\end{align}
with
\begin{align}
	\is{\eval}_{(X_1, f_{X_1}), (X_2, f_{X_2})} & : ((X_1, f_{X_1}) \Rightarrow (X_2, f_{X_2})) \times (X_1, f_{X_1}) \to (X_2, f_{X_2}) \notag                                                                \\
	                                            & =_{\text{def}} (\counit_{\comp{f_{X_1}}})_{(X_2, f_{X_2})} \circ \comp{f_{X_1}}((\counit_{\pullra{f_{X_1}}})_{\pull{f_{X_1}}(X_2, f_{X_2})}) \label{eq-eval}
\end{align}

To have a cartesian closed category in every slice, we will have to suppose the initial category \(\ccat{C}\) is \emph{locally cartesian closed} (LCC).
LCC categories are of interest on their own, because e.g. of the link they have to dependent type~\cite{Seely1984}, but whenever they have a terminal object seems to vary with the author\footnote{Compare \enquote{By convention, a locally cartesian closed category is assumed to have a terminal object, so that it is in particular cartesian closed.}~\cite[48]{Johnstone2002} with \enquote{A locally cartesian category which has a terminal object is cartesian closed.}~\cites[381--382]{Barr2012}[Proposition 13.4.6]{Barr2012}. Steve Awodey~\cite[Remark 9.21]{Awodey2010} writes it explicitly: \enquote{The reader should be aware that some authors do not require the existence of a terminal object in the definition of a locally cartesian closed category.}}.

\begin{definition}[Locally Cartesian Closed]
	\label{def:lcc}
	A category \(\ccat{C}\) is \emph{locally cartesian closed} if, equivalently,
	\begin{enumerate}
		\item it has pullbacks and every morphism is exponentiable \cite%
		      [13]{Johnstone2002}
		\item each slice category \(\slic{C}{A}\) is cartesian closed \cites[Corollary 1.5.3]{Johnstone2002}[81]{Jacobs1999}
		\item \(\ccat{C}\) has a terminal object, and for all morphism \(f : C \to D\) in \(\ccat{C}\), the composition functor (\autoref{composition-def}) \(\Sigma_f : \slic{C}{C} \to \slic{C}{D}\) has a right adjoint \(\pull{f}\), which in turns has a right adjoint \(\pullra{f}\) (\autoref{def:adjunction-to-f}).
	\end{enumerate}
\end{definition}

\section{Cartesian Structure}
\label{sec:cartesian-structure-in-slice}

\subsection{Terminal Object}
\label{sub:slice-term}

\begin{lemma}[Terminal Object]
	For all \(\ccat{C}\) and \(A\) an object of \(\ccat{C}\), \(\slic{C}{A}\) has terminal object.
\end{lemma}

\begin{proof}
	We prove that \((A, \id_A)\) is a terminal object in \(\slic{C}{A}\): let \((X, f_X)\) be an object in \(\slic{C}{A}\), we want to construct a unique \(h : (X, f_X) \to (A, \id_A)\) such that \(\id_a \circ h = f_X\).
	But \(\id_a \circ h = f_X\) implies that \(h = f_X\), and it is unique, since any other morphism \(h'\) would be such that \(\id_a \circ h' = h' = f_X = h\).
\end{proof}

Remark that the unique morphism between an object and the terminal object is given by the object itself.

\subsection{Products}
\label{sub:slice-product}

\begin{remark}[On products]
	The \enquote{spontaneous} way to define a product in \(\slic{C}{A}\) from the product in \(\ccat{C}\) does not work: suppose we define \((X, f_X) \times_{\slic{C}{A}} (Y, f_Y)\) to be \((x \times_{\ccat{C}} y), (f_X \times_{\ccat{C}} f_Y)\), as \(f_X \times_{\ccat{E}} f_Y\) is a morphism into \(A \times_{\ccat{C}} A\), it is not a morphism in \(\slic{C}{A}\).
\end{remark}

\begin{lemma}[Products]
	\label{product-def}
	If \(\ccat{C}\) has pullbacks and \(A\) is an object of \(\ccat{C}\), then \(\slic{C}{A}\) has product.
\end{lemma}

\begin{figure}
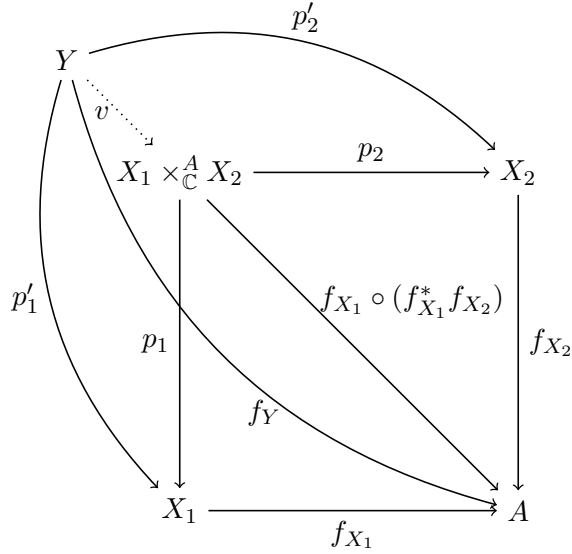

	{\centering

		\tikz{
			\node (y) at (0, 6) {\(Y\)};
			\node (pull) at (1.5, 4.5) {\(X_1 \times_{\ccat{C}}^A X_2\)};
			\node (x1) at (1.5, 0) {\(X_1\)};
			\node (x2) at (6, 4.5) {\(X_2\)};
			\node (a) at (6, 0) {\(A\)};
			\draw [->] (pull) to node[pos=0.35, right]{\(f_{X_1} \circ (f_{X_1}^*f_{X_2})\)} (a);
			\draw[->] (pull) to node[above]{\(p_2\)} (x2);
			\draw[->] (pull) to node[left]{\(p_1\)} (x1);
			\draw[->] (x1) to node[below]{\(f_{X_1}\)} (a);
			\draw[->] (x2) to node[right]{\(f_{X_2}\)} (a);
			\draw[->] (y) to [bend left] node[above]{\(p_2'\)} (x2);
			\draw[->] (y) to [bend right] node[left]{\(p_1'\)} (x1);
			\draw[->] (y) to [bend right] node[left, pos=0.65]{\(f_Y\)} (a);
			\draw[->, dotted] (y) to node[left]{\(v\)} (pull);
		}

	}
	\caption{Situation in the proof of \autoref{product-def}}
	\label{fig:product-def}
\end{figure}

\begin{proof}
	Let \((X_1, f_{X_1})\) and \((X_2, f_{X_2})\) be objects of \(\slic{C}{A}\), we write \(((X_1 \times_{\ccat{C}}^A X_2), f^*_{X_1} f_{X_2} , p_2)\) the pullback of \(f_{X_2}\) along \(f_{X_1}\) in \(\ccat{C}\) and define their product in \(\slic{C}{A}\) to be
	\((\comp{f_{X_1}}(\pull{f_{X_1}}(X_2, f_{X_2}))), f_{X_1}^*f_{X_2} , p_2)\), i.e., \((((X_1 \times_{\ccat{C}}^A X_2), f_{X_1} \circ (f_{X_1}^*f_{X_2})), f_{X_1}^*f_{X_2} , p_2)\). %
	In the following, let \(p_1 = f_{X_1}^*f_{X_2}\).

	Checking that \(((X_1 \times_{\ccat{C}}^A X_2),f_{X_1} \circ (f_{X_1}^*f_{X_2}))\) is an object in \(\slic{C}{A}\), and that \(p_1\) and \(p_2\) are morphisms in \(\slic{C}{A}\) is straightformard, and can be read from \autoref{fig:product-def}.

	For the universal property, let \((Y, f_Y)\) be an object of \(\slic{C}{A}\) and, for \(i\in \{1, 2\}\), \(p'_i : Y \to X_i\) be a morphism of \(\slic{C}{A}\), i.e.\ such that \(f_Y = f_{X_i} \circ p_i'\).
	We produce the unique \(v\) of \(\slic{C}{A}\) such that \(p_i' = p_i \circ v\) and \(f_Y = f_{X_i} \circ p_i \circ v\).
	Since \(((X_1 \times_{\ccat{C}}^A X_2), p_1, p_2)\) is the pullback of \(X_1\) and \(X_2\), we know there is a unique \(v : Y \to X_1 \times_{\ccat{C}}^A X_2\) such that \(p_i' = p_i \circ v\).
	We are left to prove that \(f_Y = f_{X_i} \circ p_i \circ v\):
	\begin{align}
		f_Y & = f_{X_i} \circ p_i' \tag{Since \(p_i'\) is a morphism in \(\slic{C}{A}\)} \\
		    & = f_{X_i} \circ p_i \circ v \tag{By definition of \(v\)}
	\end{align}\qedhere
\end{proof}

\begin{remark}[{\enquote{Altenate} product}]
	Remark that, by the universal property of the pullback, \(f_{X_1} \circ (f_{X_1}^*f_{X_2}) \cong f_{X_2} \circ (f_{X_2}^*f_{X_1})\), so that we could equivalently take the product of \((X_1, f_{X_1})\) and \((X_2, f_{X_2})\) to be \((((X_1 \times_{\ccat{C}}^A X_2),f_{X_2} \circ (f_{X_2}^*f_{X_1})), f_{X_2}^*f_{X_1} , p_2)\).

	This justifies the two presentations given in \autoref{eq-prod-1} and \autoref{eq-prod-2} and makes the product in slice categories symmetric \enquote{by construction}.
\end{remark}

\begin{remark}[On the product of morphisms]
	Given \(\is{f} : (X_1, f_{X_1}) \to (X_2, f_{X_2})\) and \(\is{g}: (Y_1, f_{Y_1}) \to (Y_2, f_{Y_2})\) two morphims in \(\slic{C}{A}\), their product \(\is{f} \times \is{g} : (X_1, f_{X_1}) \times (Y_1, f_{Y_1}) \to (X_2, f_{X_2}) \times (Y_2, f_{Y_2})\) is the only \(v\) given by the universal property of the pullback of \(f_{Y_2}\) along \(f_{X_2}\) below:

	{\centering

	\tikz{
	\node (y) at (0, 6) {\(X_1 \times_{\ccat{C}}^A Y_1\)};
	\node (pull) at (1.5, 4.5) {\(X_2 \times_{\ccat{C}}^A Y_2\)};
	\node (X2) at (1.5, 2) {\(X_2\)};
	\node (Y2) at (5, 4.5) {\(Y_2\)};
	\node (X1) at (0, 0) {\(X_1\)};
	\node (Y1) at (7, 6) {\(Y_1\)};
	\node (a) at (7, 0) {\(A\)};
	\draw[->] (pull) to node[above]{\(f_{Y_2}^*f_{X_2}\)} (Y2);
	\draw[->] (pull) to node[left]{\(f_{X_2}^*f_{Y_2}\)} (X2);
	\draw[->] (X2) to node[below]{\(f_{X_2}\)} (a);
	\draw[->] (Y2) to node[right]{\(f_{Y_2}\)} (a);
	\draw[->] (X1) to node[below]{\(f_{X_1}\)} (a);
	\draw[->] (Y1) to node[right]{\(f_{Y_1}\)} (a);
	\draw[->] (y) to %
	node[above]{\(f_{Y_1}^*f_{X_1}\)} (Y1);
	\draw[->] (y) to %
	node[left]{\(f_{X_1}^*f_{Y_1}\)} (X1);
	\draw[->, dotted] (y) to node[left]{\(v\)} (pull);
	\draw[->] (X1) to node[right]{\(f\)} (X2);
	\draw[->] (Y1) to node[left]{\(g\)} (Y2);
	}

	}

	Notice first that \(f_{X_1} = f_{X_2} \circ f\) and \(f_{Y_1} = f_{Y_2} \circ g\) since \(\is{f}\) and \(\is{g}\) are morphisms in \(\slic{C}{A}\).
	Hence, \(f_{X_2} \circ f \circ f^*_{X_1}f_{Y_1} = f_{Y_2} \circ g \circ f^*_{Y_1}f_{X_1}\), and by the universal property of the pullback of \(f_{Y_2}\) along \(f_{X_2}\), there exists a unique \(v\) such that \(f^*_{X_2}f_{Y_2} \circ v = f \circ f^*_{X_1}f_{Y_1}\) and \(f^*_{Y_2}f_{X_2} \circ v = g \circ f^*_{Y_1}f_{X_1}\).
	Hence, it follows that \(v\) is a morphism in \(\slic{C}{A}\), and we write it \(\is{f} \times \is{g}\).
\end{remark}

\subsection{Exponents}
\label{sub:slice-exponent}

\begin{lemma}[Exponents]
	\label{exponents-def}
	If for all \(f_{X_1} : X_1 \to A\), \(\pull{f_{X_1}}\) has a right adjoint, then \(\slic{C}{A}\) has exponents.
\end{lemma}

\begin{proof}
	Let \((X_2, f_{X_2})\) be an object of \(\slic{C}{A}\), we define
	\begin{itemize}
		\item \((X_1, f_{X_1}) \Rightarrow (X_2, f_{X_2})\) to be \(\pullra{f_{X_1}}(\pull{f_{X_1}}(X_2, f_{X_2}))\),
		\item the evaluation map
		      \[\is{\eval}_{(X_1, f_{X_1}), (X_2, f_{X_2})} : ((X_1, f_{X_1}) \Rightarrow (X_2, f_{X_2})) \times (X_1, f_{X_1}) \to (X_2, f_{X_2})\]
		      to be
		      \((\counit_{\comp{f_{X_1}}})_{(X_2, f_{X_2})} \circ \comp{f_{X_1}}((\counit_{\pullra{f_{X_1}}})_{\pull{f_{X_1}}(X_2, f_{X_2})})\)%
		      , where \((X_2, f_{X_2})\) (resp.\ \(\pull{f_{X_1}}(X_2, f_{X_2})\)) is the component at which the natural transformation \(\counit_{\comp{f_{X_1}}}\) (resp.\ \(\counit_{\pullra{f_{X_1}}}\)) is taken.
		\item and for all \((Z, f_Z)\)\ and \(\is{h} : (Z, f_Z) \times (X_1, f_{X_1}) \to (X_2, f_{X_2})\), the definition of \(\is{\lambda g} : (Z, f_Z) \to (X_1, f_{X_1}) \Rightarrow (X_2, f_{X_2})\) will be given below, using the properties given in \autoref{counit-comp} and \autoref{counit-pullra} of the co-units of the adjunctions given in \autoref{def:adjunction-to-f}.
	\end{itemize}

	We first check that this object and this morphism belong to \(\slic{C}{A}\):
	\begin{itemize}
		\item Since \(f_{X_1} : X_1 \to A\), \(\pull{f_{X_1}} : \slic{C}{A} \to \slic{C}{X_1}\) and \(\pullra{f_{X_1}} : \slic{C}{X_1} \to \slic{C}{A}\), we have that \((X_1, f_{X_1}) \Rightarrow (X_2, f_{X_2})\) is an object in \(\slic{C}{A}\).
		\item First, note that, by expanding the definitions of product (\autoref{product-def}) and exponent in the slice category,
		      \[((X_1, f_{X_1}) \Rightarrow (X_2, f_{X_2})) \times (X_1, f_{X_1})\]
		      is
		      \[\comp{f_{X_1}}(\pull{f_{X_1}}(\pullra{f_{X_1}}(\pull{f_{X_1}}(X_2, f_{X_2})))\]
		      and we can check that this is indeed the domain of the evaluation map.
		      Secondly, this evaluation map
		      \[(\counit_{\comp{f_{X_1}}})_{(X_2, f_{X_2})} \circ \comp{f_{X_1}}((\counit_{\pullra{f_{X_1}}})_{\pull{f_{X_1}}(X_2, f_{X_2})})\]
		      is indeed in \(\slic{C}{A}\):
		      \begin{itemize}
			      \item \(\pull{f_{X_1}}(X_2, f_{X_2})\) is in \(\slic{C}{X_1}\),
			      \item hence, \((\counit_{\pullra{f_{X_1}}})_{\pull{f_{X_1}}(X_2, f_{X_2})}\) is a morphism in \(\slic{C}{X_1}\),
			      \item and since \(\comp{f_{X_1}} : \slic{C}{X_1} \to \slic{C}{A}\), \(\comp{f_{X_1}}((\counit_{\pullra{f_{X_1}}})_{\pull{f_{X_1}}(X_2, f_{X_2})})\) is in \(\slic{C}{A}\).
			      \item for \((\counit_{\comp{f_{X_1}}})_{(X_2, f_{X_2})}\), it suffices to check that \((X_2, f_{X_2})\) is an object in \(\slic{C}{A}\), and hence that \((\counit_{\comp{f_{X_1}}})_{(X_2, f_{X_2})}\) is indeed in \(\slic{C}{A}\).
		      \end{itemize}
	\end{itemize}

	For the universal property of the evaluation map: suppose there exists \((Z, f_Z)\)\ and \(\is{h} : (Z, f_Z) \times (X_1, f_{X_1}) \to (X_2, f_{X_2})\).
	Since
	\[(Z, f_Z) \times (X_1, f_{X_1}) = \comp{f_{X_1}}(\pull{f_{X_1}}(Z, f_Z))\]
	by \autoref{counit-comp}, there exists a unique \(\is{l_1} : \pull{f_{X_1}}(Z, f_Z) \to \pull{f_{X_1}}(X_2, f_{X_2})\) such that
	\[\is{h} = (\counit_{\comp{f}})_{(X_2, f_{X_2})} \circ \comp{f_{X_1}}(\is{l_1})\]
	But then, by \autoref{counit-pullra}, there exists a unique \(\is{l_2} : (Z, f_Z) \to \pullra{f_{X_1}}(\pull{f_{X_1}}(X_2, f_{X_2}))\) such that
	\[\is{l_1} = (\counit_{\pullra{f_{X_1}}})_{(\pull{f_{X_1}}(X_2, f_{X_2}))} \circ \pull{f_{X_1}}(\is{l_2})\]
	Putting it all together, and leaving the subscripts aside, we have:
	\begin{align*}
		\is{h} & = \counit_{\comp{f_{X_1}}} \circ \comp{f_{X_1}}(\counit_{\pullra{f_{X_1}}} \circ \pull{f_{X_1}}(\is{l_2}))                                                             \\
		       & = \counit_{\comp{f_{X_1}}} \circ \comp{f_{X_1}}(\counit_{\pullra{f_{X_1}}}) \circ \comp{f_{X_1}}(\pull{f_{X_1}}(\is{l_2})) \tag{Since \(\comp{f_{X_1}}\) is a functor} \\
		       & = \is{\eval}_{(X_1, f_{X_1}), (X_2, f_{X_2})} \circ \comp{f_{X_1}}(\pull{f_{X_1}}(\is{l_2})) \tag{By definition of \(\is{\eval}\)}                                     \\
		\shortintertext{A close inspection reveals that \(\comp{f_{X_1}}(\pull{f_{X_1}}(\is{l_2}))\) and \(\is{l_2} \times \is{\id_{(X_1, f_{X_1})}}\) are actually the same morphism: \(\pull{f_{X_1}}(\is{l_2})\) and \(\is{l_2} \times \is{\id_{(X_1, f_{X_1})}}\) are both obtained as the unique morphism between \((Z, f_Z) \times (X_1, f_{X_1})\) and \(((X_1, f_{X_1}) \Rightarrow (X_2, f_{X_2})) \times (X_1, f_{X_1})\) using the universal property of the pullback \(f_{X_1}^*l_2\), and \(\comp{f_{X_1}}\) on morphisms is the identity. Hence, we get:}
		       & = \is{\eval}_{(X_1, f_{X_1}), (X_2, f_{X_2})} \circ (\is{l_2} \times \is{\id_{(X_1, f_{X_1})}})
	\end{align*}
	Hence, the universal property of the evaluation map is proven, and we let
	\[\is{\lambda(h)} = \is{l_2} \]
	which is unique by uniqueness of \(\is{l_1}\) and \(\is{l_2}\) and is a morphism in \(\slic{C}{A}\) by construction.
\end{proof}

\chapter{On Monads, Kleisli Category and Eilenberg–Moore Category}

\section{Monads}

\begin{definition}[Monad (or triple in monoid form, or Kleisli triple)~{\cites[61]{Moggi1991}[8]{Benton2000a}[5]{Benton2000b}}]
	\label{def:monad}
	A monad \(\monad{T}\) over a category \(\ccat{C}\) is a triple \((T, \unit, \mult)\), where
	\begin{itemize}
		\item \(T : \ccat{C} \to \ccat{C}\) is an endofunctor, called \emph{the carrier},
		\item \(\unit : \id_{\ccat{C}} \natt T\) is a natural transformation, called \emph{the unit},
		\item \(\mult : T^2 \natt T\) is a natural transformation, called \emph{the multiplication}
	\end{itemize}
	such that, for all object \(A\) in \(\ccat{C}\), the following commute:

	{\centering

	\tikz{
		\node (T3) at (0,2) {\(T^3A\)};
		\node (T2) at (0,0) {\(T^2A\)};
		\node (T2') at (2,2) {\(T^2A\)};
		\node (T) at (2,0) {\(TA\)};
		\draw [->] (T3) to node[above]{\(T \mult_A\)} (T2');
		\draw [->] (T3) to node[left]{\(\mult_{TA}\)} (T2);
		\draw [->] (T2) to node[below]{\(\mult_{A}\)} (T);
		\draw [->] (T2') to node[right]{\(\mult_{A}\)} (T);
	}
	\qquad
	\tikz{
		\node (T) at (0,2) {\(TA\)};
		\node (T2) at (0,0) {\(T^2A\)};
		\node (T2') at (2,2) {\(T^2A\)};
		\node (T') at (2,0) {\(TA\)};
		\draw [->] (T) to node[above]{\(\unit_{TA}\)} (T2');
		\draw [->] (T) to node[left]{\(T\unit_A \)} (T2);
		\draw [->] (T2) to node[below]{\(\mult_A\)} (T');
		\draw [->] (T2') to node[right]{\(\mult_A\)} (T');
		\draw[double distance=2pt] (T) to node [above right= 0.0cm and -0.1cm]{\(\id_{TA}\)} (T');
	}

	}
\end{definition}

The definition of Kleisli triples and monads can vary slightly, but they are in bijection~\cite[8]{Benton2000b}.

\begin{definition}[Properties of monad]
	A monad \(\monad{T} = (T, \unit, \mult)\) over a category \(\ccat{C}\) is
	\begin{description}[style=unboxed]
		\item[(Left) Strong~{\cites[74]{Moggi1991}[168]{Jacobs1991}}] if \(\ccat{C}\) is monoidal, and \((T, \lst)\) is a (left) strong functor (\autoref{def:prop-functors}), such that the following commutes:

		      {\centering

		      \tikz{
			      \node (IxX) at (0,2) {\(I \otimes X\)};
			      \node (IxTX) at (4,2) {\(I \otimes TX\)};
			      \node (TIxX) at (2,0) {\(T(I \otimes X)\)};
			      \draw [->] (IxX) to node[above]{\(\id_I \otimes \unit_X\)} (IxTX);
			      \draw [->] (IxX) to node[left]{\(\unit_{I \otimes X}\)} (TIxX);
			      \draw [->] (IxTX) to node[above right= -0.2cm and 0.4cm]{\(\lst_{I, X}\)} (TIxX);
		      }

		      \tikz{
			      \node (IxT2X) at (0,2) {\(I \otimes T^2X\)};
			      \node (TIxTX) at (4,2) {\(T(I \otimes TX)\)};
			      \node (T2IxX) at (8,2) {\(T^2(I \otimes X)\)};
			      \draw [->] (IxT2X) to node[above]{\(\lst_{I, TX}\)} (TIxTX);
			      \draw [->] (TIxTX) to node[above]{\(T\lst_{I,TX}\)} (T2IxX);
			      \node (IxTX) at (0,0) {\(I \otimes TX\)};
			      \node (TIxX) at (8,0) {\(T(I \otimes X)\)};
			      \draw [->] (IxT2X) to node[left]{\(\id_I \otimes \mult_X\)} (IxTX);
			      \draw [->] (IxTX) to node[below]{\(\lst_{I,X}\)} (TIxX);
			      \draw [->] (T2IxX) to node[right]{\(\mult_{I \otimes X}\)} (TIxX);
		      }

		      }
		      Note that if \(\ccat{C}\) is symmetric, then a
		      \emph{swapped (or twisted) strength map} \(\sst_{A,B}^l : TA \otimes B \to T(A \otimes B)\) can be defined as \(T\sym_{B, A} \circ \lst_{B,A} \circ \sym_{TA, B}\).
		\item[Right Strong] if \(\ccat{C}\) is monoidal, and \((T, \rst)\) is a right strong functor that obey similar laws.
		      Note that if \(\ccat{C}\) is symmetric, then a
		      \emph{swapped (or twisted) strength map} \(\sst_{A,B}^r : A \otimes TB \to T(A \otimes B)\) can be defined as \(T\sym_{B, A} \circ \rst_{B,A} \circ \sym_{A, TB}\).
		\item[Commutative~{\cite[203]{Jacobs2018}}] if \(\ccat{C}\) is symmetric, \(\monad{T}\) is a right and left strong monad, and the two morphisms \(\mult \circ T\sst^r \circ \rst\) and \(\mult \circ T\sst^l \circ \lst\) are equal, in which case it is named the \emph{double strength} and written \(\dst_{A,B} : TA \otimes TB \to T(A \otimes B)\).
		\item[Affine~{~\cite[Definition 1]{Jacobs2016a}}] if \(\ccat{C}\) has a terminal object \(\one\), and \(T\one \cong \one\).
	\end{description}
\end{definition}

There is a long and interesting development about right strong monads, and commutative monads, that can be found in~\cites[71]{Mulry2013}[252--257]{Maclane1971}.
Affine, commutative, and strongly affine monads are developed in~\cite{Jacobs2016a,Jacobs2016b,Jacobs1994}, but the original theory is in~\cite{Linder1979}.
An alternative definition of strong monad, involving prestrenghts and what the author calls Kleisli strength, can be found in~\cite{Mulry2013}.

\begin{definition}[Kleisli liftings~{\cite[28]{Manes1976}}]
	\label{def:liftkl}
	Given \(\monad{T} = (T, \unit, \mult)\) a monad over \(\ccat{C}\), for \(f : A \to T B\), we define \emph{the Kleisli lifting of \(f\)} to be \(\liftkl{f} = \mult_{TB} \circ Tf : TA \to TB\).
\end{definition}

\autoref{sec:cheat-monad} gathers the equalities about the monads, the left strength, and \(\monad{T}\)-algebras (whose definition follows in \autoref{sec:em}) as well as some of the equalities that can be immediately inferred from them, that we will use in the rest of this document.

\section{Kleisli Categories}
\label{sec:kleisli}

\begin{definition}[Kleisli category~{\cite[147]{Maclane1971}}]
	Given \(\monad{T} = (T, \unit, \mult)\) a monad over \(\ccat{C}\), the Kleisli category \(\ccat{C}_{\monad{T}}\) is the category whose

	\begin{description}[style=unboxed]
		\item[objects] are the objects of \(\ccat{C}\),
		\item[morphisms] are morphisms in \(\ccat{C}\) whose target is of the form \(TX\) for \(X\) in \(\ccat{C}\), i.e. \(\Hom_{\ccat{C}_{\monad{T}}}(A, B) = \Hom_{\ccat{C}}(A, TB)\),
		\item[identity] is \(\eta_A : A \to TA\),
		\item[composition] of \(f \) in \( \Hom_{\ccat{C}_{\monad{T}}}(A,B)\) and \(g\) in \( \Hom_{\ccat{C}_{\monad{T}}}(B, C)\), \(g \circ f\) in \( \Hom_{\ccat{C}_{\monad{T}}}(A,C)\) is \(\liftkl{g} \circ f :A \to TC\).
	\end{description}
\end{definition}

\begin{remark}
	For \(f \) in \( \Hom_{\ccat{C}_{\monad{T}}}(A,B)\) and \(g\) in \( \Hom_{\ccat{C}_{\monad{T}}}(B, C)\),

	\begin{enumerate}

		\item \label{rem-hash}
		      Composition with the identity behaves as expected:
		      \begin{align*}
			      \liftkl{(f \circ \unit)} & = \mult \circ Tf \circ T\unit \tag{\autoref{def:liftkl}}       \\
			                               & = f \circ \mult \circ T \unit \tag{By naturality of \(\mult\)} \\
			                               & = f \tag{\ref{m2}}
		      \end{align*}

		\item \label{kl-commutes}
		      \begin{align*}
			      \liftkl{(\liftkl{g} \circ f)} & = \liftkl{(\mult \circ T g \circ f)} \tag{\autoref{def:liftkl}}      \\
			                                    & = \mult \circ T (\mult \circ T g \circ f) \tag{\autoref{def:liftkl}} \\
			                                    & = \mult \circ T \mult \circ T^2 g \circ T f                          \\% \tag{By distributing \(T\)} \\
			                                    & = \mult \circ Tg \circ \mult \circ T f \tag{\ref{m5}}                \\
			                                    & = \liftkl{g} \circ \liftkl{f} \tag{\autoref{def:liftkl}}
		      \end{align*}

	\end{enumerate}
\end{remark}

\section{Eilenberg–Moore Categories}
\label{sec:em}

\begin{definition}[Eilenberg–Moore category]
	Given \(\monad{T} = (T, \unit, \mult)\) a monad over \(\ccat{C}\), the Eilenberg-Moore category \(\ccat{C}^{\monad{T}}\) is the category whose
	\begin{description}[style=unboxed]
		\item[objects] are \(\monad{T}\)-algebras, i.e., \(\alg{A} = (A, f_A)\) where \(A\) is the \emph{carrier}, i.e. an object in \(\ccat{C}\), and \(f_A\) is a \emph{\(\monad{T}\)-action}, i.e., a morphism \(TA \to A\) such that the following commutes:

		      {
		      \centering

		      \tikz{
			      \node (T) at (0,2) {\(A\)};
			      \node (T2') at (2,2) {\(TA\)};
			      \node (T') at (2,0) {\(A\)};
			      \draw [->] (T) to node[above]{\(\unit_{A}\)} (T2');
			      \draw [->] (T2') to node[right]{\(f_A\)} (T');
			      \draw[double distance=2pt] (T) to node [left]{\(\id_{A}\)} (T');
		      }
		      \qquad
		      \tikz{
			      \node (T) at (0,2) {\(T^2A\)};
			      \node (T2) at (0,0) {\(TA\)};
			      \node (T2') at (2,2) {\(TA\)};
			      \node (T') at (2,0) {\(A\)};
			      \draw [->] (T) to node[above]{\(Tf_A\)} (T2');
			      \draw [->] (T) to node[left]{\(\mult_A \)} (T2);
			      \draw [->] (T2) to node[below]{\(f_A\)} (T');
			      \draw [->] (T2') to node[right]{\(f_A\)} (T');
		      }

		      }

		\item[morphisms] are the \emph{\(\monad{T}\)-homomorphisms between \(\monad{T}\)-algebras}, i.e. a morphism between \(\alg{A} = (A, f_A)\) and \(\alg{B} = (B, f_B)\) is a morphism \(f : A \to B\) in \(\ccat{C}\) such that
		      \begin{align*}
			      f \circ f_A = f_B \circ Tf
		      \end{align*}

		\item[identity] is the identity on the carrier,
		\item[composition] is the composition of the underlying morphisms in \(\ccat{C}\).
	\end{description}
\end{definition}

\begin{definition}[\(\monad{T}\)-algebra homomorphism in its right-hand argument (\(\AHom\))~{\cite[192]{Levy2001}}\footnote{Thank to Paul Blain Levy for pointing out the right definition.}]
	If \(\ccat{C}\) has product and \(\monad{T}\) is a (left) strong monad on \(\ccat{C}\), then given an object \(B\) in \(\ccat{C}\), and two algebras \(\alg{A} = (A, f_A)\) and \(\alg{C} = (C, f_C)\) in \(\ccat{C}^{\monad{T}}\), we say that a morphism \(f : B \times A \to C\) in \(\ccat{C}\) is \emph{a \(\monad{T}\)-algebra homomorphism in its right-hand argument} if the following diagram commutes:

	{\centering

	\begin{tikzpicture}
		\node (BxTA) at (0, 4.5) {\(B \times TA\)};
		\node (BxA) at (0, 3) {\(B \times A\)};
		\node (TBxA) at (3, 4.5) {\(T(B \times A)\)};
		\node (TC) at (6, 4.5) {\(TC\)};
		\node (C) at (6, 3) {\(C\)};

		\draw [->] (BxTA) to node[left]{\(\id \times f_A\)} (BxA);
		\draw [->] (TC) to node[right]{\(f_C\)} (C);
		\draw [->] (BxTA) to node[above]{\(\lst\)} (TBxA);
		\draw [->] (TBxA) to node[above]{\(Tf\)} (TC);
		\draw [->] (BxA) to node[below]{\(f\)} (C);
	\end{tikzpicture}

	}

	We write \(\AHom_{\ccat{C}}(B \times \alg{A}, \alg{C})\)\footnote{However, it should be stressed that \(B \times \alg{A}\) is \emph{not} an object in \(\ccat{C}\) nor in \(\ccat{C}^{\monad{T}}\), we are just using it as a convenient notation.} to denote the subcollection of morphisms in \(\ccat{C}\) from \(B \times A\) to \(C\) that are \(\monad{T}\)-algebra homomorphisms in their right-hand arguments.
\end{definition}

\begin{lemma}\label{lem:pi_Ahom}
	For all \(D\) in \(\ccat{C}\) and \(\alg{A} = (A, f_A)\) in \(\ccat{C}^{\monad{T}}\), \(\pi_{2} : D
	\times A \to A\) is in \(\AHom_{\ccat{C}}(D \times \alg{A}, \alg{A})\).
\end{lemma}
\begin{proof}
	\(\pi_2 \circ (\id \times f_A) = f_A \circ \pi_2 = f_A \circ T\pi_2 \circ \lst\) by \ref{p4} and \ref{s3}.
\end{proof}

Finally, we note that \(\AHom\) has some nice closure properties:

\begin{lemma}[Closure properties of \(\AHom\)]\label{lem:comp_Ahom}
	Let \(D\) be in \(\ccat{C}\), \(\alg{A} = (A, f_A)\), \(\alg{C} = (C, f_C)\) in
	\(\ccat{C}^{\monad{T}}\), and \(f\) be in \(\AHom_{\ccat{C}}(D \times \alg{A}, \alg{C})\).
	\begin{enumerate}
		\item For all \(D'\) in \(\ccat{C}\) and \(g : D' \to D\), \(f \circ (g \times
		      \id)\) is in \(\AHom_{\ccat{C}}(D' \times \alg{A}, \alg{C})\).
		\item For all \(\alg{B}\) in \(\ccat{C}^{\monad{T}}\) and \(g\) in \(\AHom_{\ccat{C}}(D \times \alg{B}, \alg{A})\), the morphism \(f \circ \langle \pi_1, g\rangle\) is in
		      \(\AHom_{\ccat{C}}(D \times \alg{B}, \alg{C})\).
	\end{enumerate}
\end{lemma}

\begin{proof}
	\begin{enumerate}
		\item
		      \begin{align*}
			        & f \circ (g \times \id) \circ (\id \times f_A)                       \\
			      = & f \circ (\id \times f_A) \circ (g \times \id) \tag{\ref{p9}}        \\
			      = & f_C \circ Tf \circ \lst \circ (g \times \id) \tag{Since \(f\) is in
				      \(\AHom_{\ccat{C}}(D \times \alg{A}, \alg{C})\)}                        \\
			      = & f_C \circ Tf \circ \lst \circ (g \times T\id)                       \\
			      = & f_C \circ Tf \circ T (g \times \id) \circ \lst \tag{\ref{s4}}       \\
			      = & f_C \circ T(f \circ (g \times \id)) \circ \lst
		      \end{align*}

		\item This part of the proof has multiple steps, and requires to take associativity explicitly into account.
		      \begin{align*}
			        & f \circ \langle \pi_1, g \rangle \circ (\id \times f_B)             \\
			      = & f \circ \langle \pi_1 \circ (\id \times f_B), g\circ (\id \times
			      f_B) \rangle \tag{\ref{p1}}                                             \\
			      = & f \circ \langle \pi_1, g\circ (\id \times f_B) \rangle
			      \tag{\ref{p4}}                                                          \\
			      = & f \circ \langle \pi_1, f_A \circ Tg \circ \lst\rangle \tag{Since
				      \(g\) is in \(\AHom_{\ccat{C}}(D \times \alg{B}, \alg{A})\)}            \\
			      = & f \circ (\id \times f_A) \circ \langle \pi_1, Tg \circ \lst
			      \rangle \tag{\ref{p2}}                                                  \\
			      = & f_C \circ Tf \circ \lst \circ \langle \pi_1, Tg \circ \lst \rangle
			      \tag{Since \(f\) is in \(\AHom_{\ccat{C}}(D \times \alg{A}, \alg{C})\)} \\
			      = & f_C \circ Tf \circ T \langle \pi_1, g \rangle \circ \lst \tag{See
				      below}                                                                  \\
			      = & f_C \circ T( f \circ \langle \pi_1, g \rangle) \circ \lst
		      \end{align*}

		      We prove that \(\lst \circ \langle \pi_1, Tg \circ \lst \rangle = T \langle \pi_1, g \rangle \circ \lst\) as follows.
		      First, observe that
		      \begin{align*}
			        & \assoc \circ (\dupl \times \id)                                                                                                                           \\
			      = & \langle \pi_1 \circ \pi_1, \pi_2 \times \id \rangle \circ (\dupl \times \id) \tag{\ref{defassoc}}                                                         \\
			      = & \langle \pi_1 \circ \pi_1 \circ (\dupl \times \id), (\pi_2 \times	\id)\circ (\dupl \times \id) \rangle \tag{\ref{p2}}                                      \\
			      = & \langle \pi_1 \circ \pi_1 \circ (\langle \id, \id\rangle \times	\id), (\pi_2 \times \id)\circ (\langle \id, \id\rangle \times \id) \rangle \tag{\ref{p6}}  \\
			      = & \langle \pi_1 \circ \pi_1 \circ (\langle \id, \id\rangle \times	\id), (\pi_2 \circ \langle \id, \id\rangle) \times (\id \circ \id) \rangle \tag{\ref{p16}} \\
			      = & \langle \pi_1 \circ \langle \id, \id\rangle \circ \pi_1, (\pi_2	\circ \langle \id, \id\rangle) \times (\id \circ \id) \rangle \tag{\ref{p4}}               \\
			      = & \langle \id \circ \pi_1, \id \times \id \rangle \tag{\ref{p5}}                                                                                            \\
			      = & \langle \pi_1, \id \rangle                                                                                                                                \\
			      = & \langle \pi_1 \circ \id, \id \circ \id \rangle                                                                                                            \\
			      = & (\pi_1 \times \id) \circ \langle \id, \id \rangle \tag{p2}                                                                                                \\
			      = & (\pi_1 \times \id) \circ \dupl \tag{\ref{p6}}
		      \end{align*}
		      Hence, we get:
		      \begin{align*}
			      \lst \circ \langle \pi_1, Tg \circ \lst \rangle & = \lst \circ (\pi_1 \times (Tg\circ \lst)) \circ \dupl \tag{\ref{p19}}                                     \\
			                                                      & = \lst \circ (\id \times (Tg\circ \lst)) \circ (\pi_1 \times \id) \circ \dupl \tag{\ref{p9}}               \\
			                                                      & = \lst \circ (\id \times (Tg\circ \lst)) \circ \assoc \circ (\dupl \times \id) \tag{Previous remark}       \\
			                                                      & = \lst \circ (\id \times Tg) \circ (\id\times \lst) \circ \assoc \circ (\dupl \times \id) \tag{\ref{p15}}  \\
			                                                      & = T (\id \times g) \circ \lst \circ (\id \times \lst) \circ \assoc \circ (\dupl \times \id) \tag{\ref{s4}} \\
			                                                      & = T (\id \times g) \circ T\assoc \circ \lst \circ (\dupl \times \id) \tag{\ref{s5}}                        \\
			                                                      & = T (\id \times g) \circ T\assoc \circ \lst \circ (\dupl \times T \id)                                     \\
			                                                      & = T (\id \times g) \circ T\assoc \circ T (\dupl \times \id) \circ \lst \tag{\ref{s4}}                      \\
			                                                      & = T ((\id \times g) \circ \assoc \circ (\dupl \times \id)) \circ \lst                                      \\
			                                                      & = T((\id \times g) \circ (\pi_1 \times \id) \circ \dupl) \circ \lst	\tag{Previous remark}                   \\
			                                                      & = T((\pi_1 \times g)\circ \dupl) \circ \lst \tag{\ref{p9}}                                                 \\
			                                                      & = T \langle \pi_1 , g \rangle \circ \lst \tag{\ref{p19}}
		      \end{align*}
	\end{enumerate}
\end{proof}

\subsection{Terminal Object}

\begin{theorem}\label{thm:alg-term-obj}
	If \(\ccat{C}\) has a terminal object \(\one\), then \(\alg{\one} = (\one, !_{T\one})\) is a terminal object in \(\ccat{C}^{\monad{T}}\).%
\end{theorem}

\begin{proof}
	First, observe that \(\alg{\one} = (\one, !_{T\one})\) is an object in \(\ccat{C}^{\monad{T}}\):

	{\centering

	\begin{tikzpicture}
		\node (TT1) at (0, 3) {\(T^2\one\)};
		\node (T1) at (3, 3) {\(T\one\)};
		\node (T12) at (0, 0) {\(T\one\)};
		\node (11) at (3, 0) {\(\one\)};
		\node (12) at (5, 0) {\(\one\)};
		\draw [->] (TT1) to node[above]{\(T!_{T\one}\)}(T1);
		\draw [->] (TT1) to node[left]{\(\mult_{\one}\)} (T12);
		\draw [->] (T12) to node[below]{\(!_{T\one}\)} (11);
		\draw [->] (T1) to node[right]{\(!_{T\one}\)} (11);
		\draw [double distance=2pt] (11) to node[below]{\(\id_{\one}\)} (12);
		\draw [->] (12) to node[right]{\(\unit_{\one}\)}(T1);
	\end{tikzpicture}

	}

	all commutes because there is only one morphism from \(T^2\one\) to \(\one\), and only one morphism from \(\one\) to \(\one\) in \(\ccat{C}\).

	Given \((A, f_A)\) in \(\ccat{C}^{\monad{T}}\), we use that \(\one\) is terminal in \(\ccat{C}\) to obtain a morphism \(!_A : A \to \one\), and note that it is a morphism in \(\ccat{C}^{\monad{T}}\):

	{
	\centering

	\begin{tikzpicture}
		\node (TA) at (0, 2.5) {\(TA\)};
		\node (A) at (0, 0) {\(A\)};
		\node (T1) at (2.5, 2.5) {\(T\one\)};
		\node (1) at (2.5, 0) {\(\one\)};
		\draw [->] (TA) to node[left]{\(f_A\)} (A);
		\draw [->] (TA) to node[above]{\(T!_{A}\)} (T1);
		\draw [->] (A) to node[below]{\(!_{A}\)} (1);
		\draw [->] (T1) to node[right]{\(!_{T\one}\)} (1);
	\end{tikzpicture}

	}

	Everything commutes in this diagram because there is only one morphism
	from \(TA\) to \(\one\) in \(\ccat{C}\).
\end{proof}

\subsection{Products}

\begin{theorem}
	\label{thm:alg-product}
	If \(\ccat{C}\) has product, then
	\(\ccat{C}^{\monad{T}}\) has products, defined by \(\alg{A} \times\alg{B} = (A \times B,
	((f_A \times f_B) \circ \langle T\pi_1, T\pi_2 \rangle))\) and with
	projections \(\pi_i\) inherited from \(\ccat{C}\).
\end{theorem}
\begin{proof}
	We have to prove that
	\begin{enumerate*}
		\item that our candidate is an object in \(\ccat{C}^{\monad{T}}\),
		\item that our projections are \(\monad{T}\)-algebra homomorphisms, and
		\item that our candidate together with the projections satisfy the universal property of the product.
	\end{enumerate*}

	\begin{enumerate}
		\item We have to prove that \((f_A \times f_B) \circ \langle T\pi_1, T\pi_2 \rangle\) satisfies \ref{a1} and \ref{a2}, and we'll use that \(f_A\) and \(f_B\) satisfy them:
		      \begin{align*}
			      (f_A \times f_B) \circ \langle T\pi_1, T\pi_2 \rangle \circ \unit & = (f_A \times f_B) \circ \langle T\pi_1\circ \unit, T\pi_2\circ \unit
			      \rangle \tag{\ref{p1}}                                                                                                                                          \\
			                                                                        & =(f_A \times f_B) \circ \langle \unit \circ \pi_1, \unit \circ \pi_2 \rangle \tag{\ref{m4}} \\
			                                                                        & =\langle f_A \circ \unit \circ \pi_1, f_B \circ \unit \circ \pi_2 \rangle \tag{\ref{p2}}    \\
			                                                                        & =\langle \pi_1, \pi_2 \rangle \tag{\ref{a2}}                                                \\ &=\id \tag{\ref{p3}}
		      \end{align*}
		      \begin{align*}
			        & (f_A \times f_B) \circ \langle T\pi_1, T\pi_2 \rangle \circ T((f_A \times f_B) \circ \langle T\pi_1, T\pi_2 \rangle)                                       \\
			      = & (f_A \times f_B) \circ \langle T\pi_1\circ T(f_A \times f_B), T\pi_2\circ T(f_A \times f_B) \rangle \circ T(\langle T\pi_1, T\pi_2 \rangle) \tag{\ref{p1}} \\
			      = & (f_A \times f_B) \circ \langle T(\pi_1 \circ (f_A \times f_B)), T(\pi_2 \circ (f_A \times f_B)) \rangle \circ T(\langle T\pi_1, T\pi_2
			      \rangle)                                                                                                                                                       \\
			      = & (f_A \times f_B) \circ \langle T(f_A \circ \pi_1), T(f_B \circ \pi_2) \rangle \circ T(\langle T\pi_1, T\pi_2 \rangle) \tag{\ref{p4}}                       \\
			      = & \langle f_A \circ T(f_A \circ \pi_1), f_B \circ T(f_B \circ \pi_2)	\rangle \circ T(\langle T\pi_1, T\pi_2 \rangle) \tag{\ref{p2}}                           \\
			      = & \langle f_A \circ Tf_A \circ T\pi_1, f_B \circ T(f_B) \circ T\pi_2	\rangle \circ T(\langle T\pi_1, T\pi_2 \rangle)                                          \\
			      = & \langle f_A \circ \mult \circ T\pi_1, f_B \circ \mult \circ T\pi_2	\rangle \circ T(\langle T\pi_1, T\pi_2 \rangle) \tag{\ref{a2}}                           \\
			      = & (f_A \times f_B) \circ \langle \mult \circ T\pi_1, \mult \circ T\pi_2 \rangle \circ T(\langle T\pi_1, T\pi_2 \rangle)
			      \tag{\ref{p2}}                                                                                                                                                 \\
			      = & (f_A \times f_B) \circ \langle \mult \circ T\pi_1 \circ T(\langle T\pi_1, T\pi_2 \rangle), \mult \circ T\pi_2 \circ T(\langle T\pi_1,
			      T\pi_2 \rangle) \rangle \tag{\ref{p1}}                                                                                                                         \\
			      = & (f_A \times f_B) \circ \langle \mult \circ T(\pi_1 \circ \langle T\pi_1, T\pi_2 \rangle), \mult \circ T(\pi_2 \circ \langle T\pi_1,
			      T\pi_2 \rangle) \rangle                                                                                                                                        \\
			      = & (f_A \times f_B) \circ \langle \mult \circ T^2(\pi_1), \mult \circ T^2(\pi_2) \rangle \tag{\ref{p5}}                                                       \\
			      = & (f_A \times f_B) \circ \langle T\pi_1 \circ \mult, T\pi_2 \circ \mult \rangle \tag{\ref{m6}}                                                               \\
			      = & (f_A \times f_B) \circ \langle T\pi_1, T\pi_2\rangle \circ \mult	\tag{\ref{p1}}
		      \end{align*}
		\item We let \(\pi_1 : \alg{A}\times\alg{B} \to \alg{A}\) be the first projection, and prove that it is a morphism in \(\ccat{C}^{\monad{T}}\).
		      \begin{align*}
			      \pi_1 \circ (f_A \times f_B) \circ \langle T\pi_1, T\pi_2\rangle & =	f_A \circ \pi_1 \circ \langle T\pi_1, T\pi_2\rangle \tag{\ref{p4}} \\
			                                                                       & = f_A \circ T\pi_1\tag{\ref{p5}}
		      \end{align*}
		      We prove similarly that \(\pi_2 : \alg{A}\times\alg{B} \to \alg{B}\) is a morphism in \(\ccat{C}^{\monad{T}}\).
		\item We now have to prove the universal property of that product,	i.e., that for all \(\alg{C} = (C, f_C)\) such that there exist
		      \(f: \alg{C} \to \alg{A}\) and \(g: \alg{C} \to \alg{B}\), there exists a unique \(h : \alg{C} \to \alg{A} \times \alg{B}\) such that \(f = \pi_1 \circ h\) and \(g = \pi_2 \circ h\).
		      A picture at the end of this part depicts the situation at the end of this proof.

		      Let \(h = \langle f, g \rangle\) be the mediating morphism into the product.
		      We first prove it is a morphism in \(\ccat{C}^{\monad{T}}\):
		      \begin{align*}
			      \langle f, g \rangle \circ f_C & = \langle f \circ f_C, g \circ
			      f_C\rangle \tag{\ref{p1}}                                                                                                                                      \\
			                                     & = \langle f_A \circ Tf, f_B \circ Tg \rangle \tag{Since \(f\) and \(g\)
			      are morphisms in \(\ccat{C}^{\monad{T}}\)}                                                                                                                     \\
			                                     & = (f_A \times f_B) \circ \langle Tf, Tg\rangle \tag{\ref{p2}}                                                                 \\
			                                     & = (f_A \times f_B) \circ \langle T(\pi_1 \circ \langle f, g\rangle), T(\pi_2 \circ \langle f, g\rangle)\rangle \tag{\ref{p5}} \\
			                                     & = (f_A \times f_B) \circ \langle T\pi_1 \circ T\langle f, g\rangle, T\pi_2 \circ T\langle f, g\rangle\rangle                  \\
			                                     & = (f_A \times f_B) \circ \langle T\pi_1, T\pi_2 \rangle \circ T\langle f, g\rangle \tag{\ref{p1}}
		      \end{align*}
		      That \(f = \pi_1 \circ h\) follows from \(\pi_1 \circ \langle f, g
		      \rangle = f\), similarly for \(g = \pi_2 \circ h\).

		      If there were another morphism \(h'\) with the same properties, we could
		      get a \(h': C \to A \times B\) that would contradict the uniqueness of
		      \(\langle f, g\rangle\) with respect to the product in \(\ccat{C}\). The
		      picture we obtain is the following:

		      
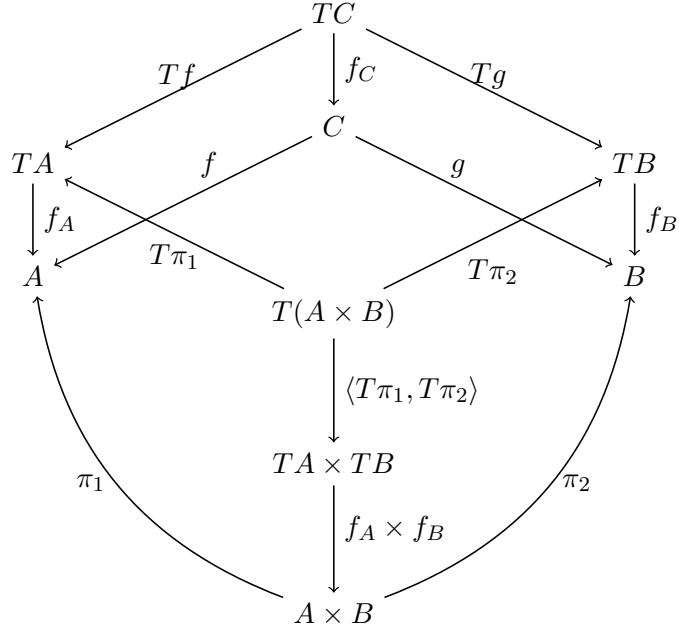
{\centering

		      \begin{tikzpicture}
			      \node (TC) at (4, 6) {\(TC\)};
			      \node (C) at (4, 4.5) {\(C\)};
			      \draw[->] (TC) to node[right]{\(f_C\)} (C);
			      \node (TAxB) at (4, 2) {\(T(A \times B)\)};
			      \node (TAxTB) at (4, 0) {\(TA \times TB\)};
			      \node (AxB) at (4, -2) {\(A \times B\)};
			      \draw[->] (TAxB) to node[right]{\(\langle T\pi_1, T\pi_2\rangle\)} (TAxTB);
			      \draw[->] (TAxTB) to node[right, pos=0.4]{\(f_A \times f_B\)} (AxB);
			      \node (TA) at (0, 4) {\(TA\)};
			      \node (A) at (0, 2.5) {\(A\)};
			      \draw[->] (TA) to node[right]{\(f_A\)} (A);
			      \draw[->] (TC) to node[left, pos=0.4]{\(Tf\)} (TA);
			      \draw[->] (C) to node[above, pos=0.4]{\(f\)} (A);
			      \draw[->] (TAxB) to node[below]{\(T\pi_1\)} (TA);
			      \draw[->] (AxB) to [bend left] node[left]{\(\pi_1\)} (A);
			      \node (TB) at (8, 4) {\(TB\)};
			      \node (B) at (8, 2.5) {\(B\)};
			      \draw[->] (TB) to node[right]{\(f_B\)} (B);
			      \draw[->] (TC) to node[right, pos=0.4]{\(Tg\)} (TB);
			      \draw[->] (C) to node[above, pos=0.4]{\(g\)} (B);
			      \draw[->] (TAxB) to node[below=0.2cm]{\(T\pi_2\)} (TB);
			      \draw[->] (AxB) to [bend right] node[right]{\(\pi_2\)} (B);
		      \end{tikzpicture}

		      }
	\end{enumerate}

	We also note that the duplication \(\dupl\) is easily defined as a
	morphism in \(\ccat{C}^{\monad{T}}\), since

		{\centering

			\begin{tikzpicture}
				\node (TA) at (0, 4.5) {\(TA\)};
				\node (A) at (0, 3) {\(A\)};
				\node (TAxTA) at (5, 4.5) {\(TA \times TA\)};
				\node (AxA) at (5, 3) {\(A \times A\)};

				\draw [->] (TA) to node[left]{\(f_A\)} (A);
				\draw [->] (TAxTA) to node[right]{\((f_A \times f_A) \circ \langle T\pi_1, T\pi_2 \rangle \)} (AxA);
				\draw [->] (TA) to node[above]{\(T\dupl\)} (TAxTA);
				\draw [->] (A) to node[above]{\(\dupl\)} (AxA);
			\end{tikzpicture}

		}

	commutes trivially:
	\begin{align*}
		(f_A \times f_A) \circ \langle T\pi_1, T\pi_2 \rangle \circ T\dupl
		 & = (f_A \times f_A) \circ \langle T\pi_1 \circ T\dupl, T\pi_2 \circ T\dupl\rangle \tag{\ref{p1}} \\
		 & = (f_A \times f_A) \circ \langle T(\pi_1 \circ \dupl), T(\pi_2 \circ \dupl)\rangle              \\
		 & = (f_A \times f_A) \circ \langle T\id, T\id\rangle
		\tag{\ref{p13}}                                                                                    \\
		 & = (f_A \times f_A) \circ \langle \id, \id\rangle                                                \\
		 & = (f_A \times f_A) \circ \dupl \tag{\ref{p6}}                                                   \\
		 & = \dupl \circ f_A\tag{\ref{p14}}
	\end{align*}
\end{proof}

\subsection{Exponent-like Structures}

Eilenberg–Moore categories do not have exponents, but can be endowed with two structures that share similarities with exponents.
Below, they are named \enquote{internal} and \enquote{external}, but no \enquote{canonical} name for them is known.
The \enquote{external} is used, for instance, in \cite[Lemma 3.1.]{Mogelberg2009}.

\subsubsection{\enquote{Internal Exponents}}

\begin{theorem}\label{thm:expoalg}
	If \(\ccat{C}\) is cartesian closed and \(\monad{T}\) is (left) strong, letting \(\alg{A} = (A, f_A)\) be an object in \(\ccat{C}^{\monad{T}}\) and \(B\) be an object in \(\ccat{C}\), then \(B \texpo \alg{A} \eqdef (B \expo A, \lambda(f_A \circ T\ev \circ T\sym \circ \lst \circ \sym))\) is an object in \(\ccat{C}^{\monad{T}}\).
\end{theorem}

\begin{proof}
	We need to show that \(\lambda(f_A \circ T\ev \circ T\sym \circ \lst \circ \sym) : T(B \expo A) \to B \expo A\) satisfies (\ref{a1}) and (\ref{a2}). We will use that, since \((A, f_A)\) is an object in \(\ccat{B}^{\monad{T}}\), \(f_A\) satisfies them.
	\begin{align*}
		\lambda (f_A \circ T\ev \circ T\sym \circ \lst \circ \sym) \circ \unit
		 & = \lambda (f_A \circ T\ev \circ T\sym \circ \lst \circ \sym \circ (\unit \times \id)) \tag{\ref{e1}} \\
		 & = \lambda (f_A \circ T\ev \circ T\sym \circ \lst \circ (\id \times	\unit)\circ \sym)\tag{\ref{p11}}   \\
		 & = \lambda (f_A \circ T\ev \circ T\sym \circ \unit \circ \sym)	\tag{\ref{s1}}                          \\
		 & = \lambda (f_A \circ T\ev \circ \unit \circ \sym \circ \sym) \tag{\ref{m4}}                          \\
		 & = \lambda (f_A \circ T\ev \circ \unit) \tag{\ref{p12}}                                               \\
		 & = \lambda (f_A \circ \unit \circ \ev) \tag{\ref{m4}}                                                 \\\
		 & = \lambda (\ev) \tag{\ref{a1}}                                                                       \\
		 & = \id \tag{\ref{e2}}
	\end{align*}

	\begin{align*}
		  & \lambda (f_A \circ T\ev \circ T\sym \circ \lst \circ \sym) \circ \mult                                                                                      \\
		= & \lambda (f_A \circ T\ev \circ T\sym \circ \lst \circ \sym \circ (\mult \times \id)) \tag{\ref{e1}}                                                          \\
		= & \lambda (f_A \circ T\ev \circ T\sym \circ \lst \circ (\id \times	\mult) \circ \sym) \tag{\ref{p11}}                                                          \\
		= & \lambda (f_A \circ T\ev \circ T\sym \circ \mult \circ T\lst \circ \lst \circ \sym) \tag{\ref{s2}}                                                           \\
		= & \lambda (f_A \circ T\ev \circ \mult \circ T^2(\sym) \circ T\lst \circ \lst \circ \sym) \tag{\ref{m6}}                                                       \\
		= & \lambda (f_A \circ \mult \circ T^2(\ev) \circ T^2(\sym) \circ T\lst \circ \lst \circ \sym) \tag{\ref{m6}}                                                   \\
		= & \lambda (f_A \circ Tf_A \circ T^2(\ev) \circ T^2(\sym) \circ T\lst \circ \lst \circ \sym) \tag{\ref{a2}}                                                    \\
		= & \lambda (f_A \circ Tf_A \circ T^2(\ev) \circ T^2(\sym) \circ T\lst \circ T\sym \circ T\sym \circ \lst \circ \sym) \tag{\ref{p12}}                           \\
		= & \lambda (f_A \circ T(f_A \circ T\ev \circ T\sym \circ \lst \circ \sym) \circ T\sym \circ \lst \circ \sym)                                                   \\
		= & \lambda (f_A \circ T(\ev \circ (\lambda (f_A \circ T\ev \circ T\sym \circ \lst \circ \sym) \times \id)) \circ T\sym \circ \lst \circ \sym) \tag{\ref{e3}}   \\
		= & \lambda (f_A \circ T\ev \circ T (\lambda (f_A \circ T\ev \circ T\sym \circ \lst \circ \sym) \times \id)\circ T\sym \circ \lst \circ \sym)                   \\
		= & \lambda (f_A \circ T\ev \circ T\sym \circ T (\id \times (\lambda	(f_A \circ T\ev \circ T\sym \circ \lst \circ \sym))) \circ \lst \circ \sym) \tag{\ref{p11}} \\
		= & \lambda (f_A \circ T\ev \circ T\sym \circ \lst \circ (\id \times T(\lambda (f_A \circ T\ev \circ T\sym \circ \lst \circ \sym)))\circ \sym) \tag{\ref{s4}}   \\
		= & \lambda (f_A \circ T\ev \circ T\sym \circ \lst \circ \sym \circ (T(\lambda (f_A \circ T\ev \circ T\sym \circ \lst \circ \sym)) \times \id)) \tag{\ref{p11}} \\
		= & \lambda (f_A \circ T\ev \circ T\sym \circ \lst \circ \sym) \circ T(\lambda (f_A \circ T\ev \circ T\sym \circ \lst \circ \sym)) \tag{\ref{e1}}
	\end{align*}
\end{proof}

Note that if \(\forg : \ccat{C}^{\monad{T}} \to \ccat{C}\) is the forgetful functor associated with \(\monad{T}\), then
\begin{align*}
	\forg (B \texpo \alg{A}) & = \forg (B \expo A, \lambda(f_A \circ T\ev \circ T\sym \circ \lst \circ \sym)) \\
	                         & = B \expo A                                                                    \\
	                         & = B \expo (\forg\, (A, f_A))                                                   \\
	                         & = B \expo (\forg \,\alg{A})
\end{align*}

\begin{remark}
	Note that \(\ev \circ \sym\) is in \(\AHom_{\ccat{C}}(A \times A \texpo \alg{B},	\alg{B})\):
	\begin{align*}
		  & \ev \circ \sym \circ (\id \times \lambda(f_B \circ T\ev \circ T\sym \circ \rst \circ \sym)) \\
		= & \ev \circ (\lambda(f_B \circ T\ev \circ T\sym \circ \rst \circ \sym) \times \id) \circ \sym \\
		= & f_B \circ T\ev \circ T\sym \circ \rst \circ \sym \circ \sym \tag{\ref{e3}}                  \\
		= & f_B \circ T\ev \circ T\sym \circ \rst \tag{\ref{p12}}
	\end{align*}
\end{remark}

\subsubsection{\enquote{External Exponents}}
\begin{theorem}\label{thm:iso2}
	If \(\ccat{C}\) has all equalizers, exponents and products, for every object \(C\) in \(\ccat{C}\), and algebras \(\alg{A} = (A, f_A)\),
	\(\alg{B} = (B, f_B)\) in \(\ccat{C}^{\monad{T}}\), there exists an object \(\alg{A} \mm \alg{B}\) in \(\ccat{C}\) such that \(\AHom_{\ccat{C}}(C \times \alg{A}, \alg{B}) \cong \Hom_{\ccat{C}}(C, \alg{A} \mm \alg{B})\) .
\end{theorem}
\begin{proof}
	We have to
	\begin{enumerate}
		\item give the definition of the object \(\alg{A} \mm \alg{B}\),
		\item construct \(\Theta : \AHom_{C}(C \times \alg{A},\alg{B}) \to \Hom_{\ccat{C}}(C, \alg{A} \mm \alg{B})\),
		\item construct	\(\Omega : \Hom_{\ccat{C}}(C, \alg{A} \mm \alg{B}) \to \AHom_{C}(C \times \alg{A}, \alg{B})\) and
		\item prove that \(\Theta \circ \Omega = \Omega \circ \Theta = \id\).
	\end{enumerate}
	\begin{enumerate}
		\item Let \((\alg{A} \mm \alg{B}, \meq)\) be the equalizer of \(\lambda(\ev \circ (\id \times f_A)) : (A \expo B) \to (TA \expo B)\) and \(\lambda(f_B \circ T \ev \circ \rst) : (A \expo B) \to (TA \expo B)\):

		      {\centering

		      \begin{tikzpicture}
			      \node (eq) at (1, 0) {\(\alg{A} \mm \alg{B}\)};
			      \node (AtoB) at (3, 0) {\(A \expo B\)};
			      \node (TAtoB) at (8, 0) {\(TA \expo B\)};
			      \draw[->, bend left] (AtoB) to node[above]{\(\lambda(\ev \circ (\id \times f_A))\)}(TAtoB);
			      \draw[->, bend right] (AtoB) to node[below]{\(\lambda(f_B \circ T\ev \circ \rst)\)} (TAtoB);
			      \draw[->] (eq) to node[above]{\(\meq\)} (AtoB);
		      \end{tikzpicture}

		      }

		\item Given \(f\) in \(\AHom_{\ccat{C}}(C \times \alg{A}, \alg{B})\), we let \(\Theta f\) be the morphism \(m_f : C \to \alg{A} \mm \alg{B}\) given
		      by \((\alg{A} \mm \alg{B}, \meq)\):

		      {\centering

		      \begin{tikzpicture}
			      \node (C) at (0, 2) {\(C\)};
			      \node (eq) at (0, 0) {\(\alg{A} \mm \alg{B}\)};
			      \node (AtoB) at (3, 0) {\(A \expo B\)};
			      \node (TAtoB) at (6, 0) {\(TA \expo B\)};
			      \draw[->, bend left] (AtoB) to (TAtoB);
			      \draw[->, bend right] (AtoB) to (TAtoB);
			      \draw[->, dashed] (C) to node[left]{\(\exists ! \,m_f\)} (eq);
			      \draw[->] (eq) to node[above]{\(\meq\)} (AtoB);
			      \draw[->] (C) to node[above]{\(\lambda f\)} (AtoB);
		      \end{tikzpicture}

		      }

		      We verify that the property of the equalizer %
		      can indeed be used:
		      \begin{align*}
			        & \lambda(\ev \circ (\id \times f_A)) \circ \lambda f                                              \\
			      = & \lambda(\ev \circ (\id \times f_A) \circ (\lambda f \times \id))\tag{\ref{e1}}                   \\
			      = & \lambda(\ev \circ (\lambda f \times \id) \circ (\id \times f_A))\tag{\ref{p9}}                   \\
			      = & \lambda(f \circ (\id \times f_A))\tag{\ref{e3}}                                                  \\
			      = & \lambda(f_B \circ Tf \circ \rst)\tag{Since \(f\) is in \(\AHom_{C}(C \times \alg{A}, \alg{B})\)} \\
			      = & \lambda(f_B \circ T(\ev \circ (\lambda f \times \id)) \circ \rst)\tag{\ref{e3}}                  \\
			      = & \lambda(f_B \circ T \ev \circ T(\lambda f \times \id) \circ \rst)                                \\
			      = & \lambda(f_B \circ T \ev \circ \rst \circ (\lambda f \times T\id)) \tag{\ref{s4}}                 \\
			      = & \lambda(f_B \circ T \ev \circ \rst \circ (\lambda f \times \id))                                 \\
			      = & \lambda(f_B \circ T\ev \circ \rst) \circ \lambda f \tag{\ref{e1}}
		      \end{align*}
		\item Given \(g\) in \(\Hom_{\ccat{C}}(C, \alg{A} \mm \alg{B})\), we define
		      \(\Omega g\) to be \(\lambda^{-1}(\meq \circ g)\). We prove that it is
		      a morphism in \(\AHom_{\ccat{C}}(C \times \alg{A}, \alg{B})\) using that
		      \(\meq\) is the equalizer of \(\lambda(\ev \circ (\id \times f_A))\) and
		      \(\lambda(f_B \circ T \ev \circ \rst)\):
		      \begin{align*}
			               & \lambda(\ev \circ (\id \times f_A)) \circ \meq = \lambda(f_B \circ T\ev \circ \rst) \circ \meq                                             \\
			      \implies & \lambda(\ev \circ (\id \times f_A)) \circ \meq \circ g= \lambda(f_B \circ T\ev \circ \rst) \circ \meq \circ g                              \\
			      \implies & \lambda^{-1}(\lambda(\ev \circ (\id \times f_A)) \circ \meq \circ g) = \lambda^{-1}(\lambda(f_B \circ T\ev \circ \rst) \circ \meq \circ g) \\
			      \implies & \lambda^{-1}(\lambda(\ev \circ (\id \times f_A))) \circ ((\meq \circ g) \times \id)                                                        \\
			               & \hspace{2em} = \lambda^{-1}(\lambda(f_B \circ T\ev	\circ \rst)) \circ ((\meq \circ g)\times \id) \tag{\ref{e4}}                             \\
			      \implies & \ev \circ (\id \times f_A) \circ ((\meq \circ g) \times \id) = f_B \circ T\ev \circ \rst \circ ((\meq \circ g)\times \id)\tag{\ref{e6}}    \\
			      \implies & \ev \circ ((\meq \circ g) \times \id)\circ (\id \times f_A)
			      = f_B \circ T\ev \circ \rst \circ ((\meq \circ g)\times T\id) \tag{\ref{p9}}                                                                          \\
			      \implies & \ev \circ ((\meq \circ g) \times \id)\circ (\id \times f_A) = f_B \circ T\ev \circ T((\meq \circ g)\times \id) \circ \rst \tag{\ref{s4}}   \\
			      \implies & \ev \circ ((\meq \circ g) \times \id)\circ (\id \times f_A)
			      = f_B \circ T (\ev \circ ((\meq \circ g)\times \id)) \circ \rst                                                                                       \\
			      \implies & \ev \circ (\lambda(\lambda^{-1}(\meq \circ g)) \times \id)\circ (\id \times f_A)                                                           \\
			               & \hspace{2em}= f_B \circ T (\ev \circ (\lambda(\lambda^{-1}(\meq \circ g))\times \id)) \circ \rst\tag{\ref{e5}}                             \\
			      \implies & \lambda^{-1}(\meq \circ g) \circ (\id \times f_A) = f_B
			      \circ T (\lambda^{-1}(\meq \circ g)) \circ \rst\tag{\ref{e3}}
		      \end{align*}
		\item Given \(g\) in \(\Hom_{\ccat{C}}(C, \alg{A} \mm \alg{B})\), \(\Theta(\Omega g) = \Theta (\lambda^{-1}(\meq \circ g))\) is the unique morphism \(m_{\lambda^{-1}(\meq \circ g))}\):

		      {\centering

		      \begin{tikzpicture}
			      \node (C) at (0, 2) {\(C\)};
			      \node (eq) at (0, 0) {\(\alg{A} \mm \alg{B}\)};
			      \node (AtoB) at (3, 0) {\(A \expo B\)};
			      \node (TAtoB) at (5, 0) {\(TA \expo B\)};
			      \draw[->, bend left] (AtoB) to (TAtoB);
			      \draw[->, bend right] (AtoB) to (TAtoB);
			      \draw[->, dashed] (C) to node[left]{\(\exists ! \,m_{\lambda^{-1}(\meq \circ g))}\)} (eq);
			      \draw[->] (eq) to node[above]{\(\meq\)} (AtoB);
			      \draw[->] (C) to node[right]{\(\lambda(\lambda^{-1}(\meq \circ g)))\)} (AtoB);
		      \end{tikzpicture}

		      }

		      But since \(\lambda(\lambda^{-1}(\meq \circ g))) = \meq \circ g\), we
		      have that \(m_{\lambda^{-1}(\meq \circ g))} = g\). Moreover, given \(f\)
		      in \(\AHom_{\ccat{C}}(C \times \alg{A}, \alg{B})\), \(\Omega( \Theta f) =
		      \lambda^{-1}(\meq \circ \Theta f )\) where \(\Theta f\) is the unique
		      morphism \(m_f\) such that \(\meq \circ m_f = \lambda f\). Hence,
		      \(\Omega( \Theta f) = \lambda^{-1}(\meq \circ m_f) =
		      \lambda^{-1}(\lambda f) = f\).

		      We conclude that \(\Theta \circ \Omega = \Omega \circ \Theta = \id\). \qedhere
	\end{enumerate}
\end{proof}

\begin{remark}
	Note that \(\ev_{\mm} \eqdef (\meq \times \id) \circ \ev\) is in \(\AHom_{\ccat{C}}(\alg{A} \mm \alg{B}\times \alg{A}, \alg{B})\): \begin{align*}
		\ev \circ (\meq \times \id) \circ (\id \times f_A) & = \ev \circ (\id \times f_A) \circ (\meq \times \id) \tag{\ref{p9}}                                                             \\
		                                                   & = \lambda^{-1}(\lambda(\ev \circ (\id \times f_A) \circ (\meq \times \id))\tag{\ref{e6}}                                        \\
		                                                   & = \lambda^{-1}(\lambda(\ev \circ (\id \times f_A)) \circ \meq)\tag{\ref{e1}}                                                    \\
		                                                   & = \lambda^{-1}(\lambda(f_B \circ T\ev \circ \rst) \circ \meq)\tag{Since \(\meq\) is the equalizer of \( \lambda (\ev \circ (\id
			\times f_A)) \) and \( \lambda(f_B \circ T\ev \circ \rst)\)}                                                                                                                         \\
		                                                   & = \lambda^{-1}(\lambda(f_B \circ T\ev \circ \rst \circ (\meq \times	\id))\tag{\ref{e1}}                                          \\
		                                                   & = f_B \circ T\ev \circ \rst \circ (\meq\times \id) \tag{\ref{e6}}                                                               \\
		                                                   & = f_B \circ T\ev \circ T(\meq\times \id)\circ \rst \tag{\ref{s4}}
	\end{align*}
\end{remark}

\subsubsection{Connecting the Internal and the External Exponents}

It is conjectured that for every object \(C\) in \(\ccat{C}\), and all algebras \(\alg{A} = (A, f_A)\), \(\alg{B} = (B, f_B)\) in \(\ccat{C}^{\monad{T}}\),
\[\Hom_{\ccat{C}}(C, \alg{A} \mm \alg{B}) \cong \Hom_{\ccat{C}^{\monad{T}}}(\alg{A}, C \texpo \alg{B})\]

However, the proof attempts require to be explicit about the associativity, and seemed doubtful.

\printbibliography[heading=bibintoc]

\clearpage

\appendix

\chapter{Cheat Sheets}
\label{cheatsheet}
\section{Cartesian Structure}
\label{sec:cheat-cartesian}
\begin{description}[style=unboxed]
	\item[Product]
	      \begin{subequations}
		      \renewcommand{\theequation}{p\textsubscript{\arabic{equation}}}
		      \begin{align}
			      \langle f, g\rangle                         & =_{\text{def}} (f \times g) \circ \dupl \label{p19}   \\
			      \sym                                        & =_{\text{def}} \pi_2 \times \pi_1                     \\
			      \dupl                                       & =_{\text{def}} \langle \id, \id \rangle \label{p6}    \\
			      \notag                                                                                              \\
			      (f_1 \times g_1) \circ (f_2 \times g_2)     & = (f_1 \circ f_2) \times (g_1 \times g_2) \label{p16} \\
			      \langle f, g \rangle \circ h                & = \langle f \circ h, g \circ h\rangle \label{p1}      \\
			      (f \times g) \circ \langle h_1, h_2 \rangle & = \langle f \circ h_1, g \circ h_2\rangle \label{p2}  \\
			      \langle \pi_1, \pi_2\rangle                 & = \id \label{p3}                                      \\
			      \pi_i \circ (f_1 \times f_2)                & = f_i \circ \pi_i \label{p4}                          \\
			      \pi_i \circ \langle f_1, f_2\rangle         & = f_i \label{p5}                                      \\
			      f \circ \pi_2                               & = \pi_2 \circ (\id \times f) \label{p7}               \\
			      f \circ \pi_1                               & = \pi_1 \circ (f \times \id) \label{p8}               \\
			      f \times g                                  & = (f \times \id) \circ (\id \times g) \notag          \\
			                                                  & = (\id \times g) \circ (f \times \id) \label{p9}      \\
			      (f \circ g) \times \id                      & = (f \times \id) \circ (g\times \id) \label{p15}      \\
			      (f \times g) \circ \sym                     & = \sym \circ (g \times f) \label{p11}                 \\
			      \sym \circ \sym                             & = \id \label{p12}                                     \\
			      \pi_i \circ \dupl                           & = \id \label{p13}                                     \\
			      (f \times f) \circ \dupl                    & = \dupl \circ f \label{p14}
		      \end{align}
	      \end{subequations}
	\item[Exponents]
	      \begin{subequations}
		      \renewcommand{\theequation}{e\textsubscript{\arabic{equation}}}
		      \begin{align}
			      \lambda f \circ g                & = \lambda(f \circ (g \times \id)) \label{e1}      \\
			      \lambda \ev                      & = \id \label{e2}                                  \\
			      \ev \circ (\lambda f \times \id) & = f\label{e3}                                     \\
			      \lambda^{-1}(f \circ g)          & = (\lambda^{-1}f) \circ ( g \times \id)\label{e4} \\
			      \lambda \lambda^{-1} f           & = f \label{e5}                                    \\
			      \lambda^{-1} \lambda f           & = f \label{e6}                                    \\
			      f = g                            & \iff \lambda^{-1}f = \lambda^{-1} g \label{e7}    \\
			      f = g                            & \iff \lambda f = \lambda g \label{e8}             \\
			      \ev \circ (f \times \id)         & = \lambda^{-1}(f) \label{e9}
		      \end{align}
	      \end{subequations}
	\item[Associativity]
	      \begin{subequations}
		      \renewcommand{\theequation}{as\textsubscript{\arabic{equation}}}
		      \begin{align}
			      \langle \pi_1 \circ \pi_1, \pi_2 \times \id \rangle & \eqdef \assoc \label{defassoc}                                         \\
			      \langle \id \times \pi_1, \pi_2 \circ \pi_2 \rangle & \eqdef \assoc^{-1} \label{defassocinv}                                 \\
			      \notag                                                                                                                       \\
			      \assoc^{-1} \circ \assoc                            & = \id \label{assoc6}                                                   \\
			      \assoc \circ \assoc^{-1}                            & = \id \label{assoc7}                                                   \\
			      \assoc \circ \sym \circ \assoc                      & = (\id \times \sym) \circ \assoc \circ (\sym \times \id)\label{assoc1} \\
			      (f_1 \times (f_2 \times f_3)) \circ \assoc          & = \assoc \circ ((f_1 \times f_2) \times f_3) \label{assoc3}            \\
			      \assoc^{-1} \circ (f_1 \times (f_2 \times f_3))     & = ((f_1 \times f_2) \times f_3) \circ \assoc^{-1}\label{assoc4}
		      \end{align}
	      \end{subequations}
\end{description}
\section{Monadic Structure}
\label{sec:cheat-monad}
\begin{description}[style=unboxed]
	\item[Monad]
	      \begin{subequations}
		      \renewcommand{\theequation}{m\textsubscript{\arabic{equation}}}
		      \begin{align}
			      \mult \circ \mult   & = \mult \circ T\mult \label{m1} \\
			      \mult \circ T \unit & = \id \label{m2}                \\
			      \mult \circ \unit   & = \id \label{m3}                \\
			      \unit \circ f       & = Tf \circ \unit \label{m4}     \\
			      Tf \circ \mult      & = T\mult \circ T^2 f \label{m5} \\
			      Tf \circ \mult      & =\mult \circ T^2 f\label{m6}
		      \end{align}
	      \end{subequations}
	\item[(Left) Strength]
	      \begin{subequations}
		      \renewcommand{\theequation}{s\textsubscript{\arabic{equation}}}
		      \begin{align}
			      \lst \circ (\id \times \unit) & = \unit \label{s1}                                   \\
			      \lst \circ (\id \times \mult) & = \mult \circ T\lst \circ \lst\label{s2}             \\
			      T\pi_2 \circ \lst             & = \pi_2\label{s3}                                    \\
			      \lst \circ (f \times Tg)      & = T(f \times g) \circ \lst \label{s4}                \\
			      T\assoc \circ \lst            & = \lst\circ (\id \times \lst) \circ \assoc\label{s5}
		      \end{align}
	      \end{subequations}
	\item[Algebras]
	      \begin{subequations}
		      {
			      \renewcommand{\theequation}{al\textsubscript{\arabic{equation}}}
			      \begin{align}
				      f_A \circ \unit & = \id \label{a1}            \\
				      f_A \circ \mult & = f_A \circ Tf_A \label{a2}
			      \end{align}
		      }
	      \end{subequations}
\end{description}

\end{document}